\documentclass[11pt]{article}
 
\usepackage{amsthm, amsfonts, amsxtra, amssymb, amscd}

\usepackage{amsfonts,amsmath,amssymb,amscd,amsthm}
\usepackage[english]{babel} 
\usepackage[all]{xy}
\usepackage[backref, colorlinks, linktocpage, citecolor = blue, linkcolor = blue]{hyperref}

\usepackage{amsfonts,graphics, epic}
 \usepackage{graphicx}
\usepackage{multicol}
 \usepackage{tikz}
\usepackage{psfrag}
 \usepackage{young}
 \usepackage{ytableau}
\usepackage[enableskew]{youngtab}
\usepackage{tikz}

\usepackage{amsmath, amssymb, amsthm}
\usepackage{mathtools}  % extra mathtools, some things fixes UMS math bugs
\usepackage{bm}         % bold math
\usepackage{dsfont}	% Double stroke roman fonts, e.g. \mathds{1}
\usepackage{ytableau}   % for Young tableaux
\usepackage{microtype}  % Subliminal refinements towards typographical perfection. 
\usepackage{placeins}

\newtheorem*{lemma*}{Lemma}
\newtheorem{lemma}[subsection]{Lemma}
\newtheorem*{theorem*}{Theorem}
\newtheorem{theorem}[subsection]{Theorem}
\newtheorem*{proposition*}{Proposition}
\newtheorem{proposition}[subsection]{Proposition}

\newtheorem*{corollary*}{Corollary}

\theoremstyle{definition}
\newtheorem*{definition*}{Definition}

\newtheorem*{example*}{Example}

\newtheorem{exercise}[subsection]{Exercise}

\theoremstyle{remark}
\newtheorem*{remark*}{Remark}
\newtheorem{remark}[subsection]{Remark}
 \newtheorem{definition}[subsection]{Definition}

\renewcommand{\phi}{\varphi}

\newcommand{\be}{\begin{enumerate}}
\newcommand{\ee}{\end{enumerate}}

\title{ Tensor fundamental theorems of invariant theory}
\author{Claudio Procesi}

\begin{document}\maketitle
\begin{abstract}
The aim of this paper is to establish a   first and second fundamental theorem for $GL(V)$ equivariant polynomial maps from $k$--tuples of matrix variables $End(V)^{  k}  $  to tensor spaces $End(V)^{  \otimes n}$, in the spirit of   H. Weyl's book {\em The classical groups} \cite{Weyl} and of symbolic algebra.

\end{abstract}\tableofcontents

\section{Introduction} In this paper  $V$ denotes a vector space  over  a field $F$  of characteristic 0   of dimension $d$. In fact since all the formulas developed have rational coefficients it is enough to assume $F=\mathbb Q$ and $End(V)=M_d(F)=M_d,$  matrices in a basis.

The aim of this paper is to establish a   first and second fundamental theorem (FFT and SFT for short), in the spirit of   H. Weyl's book {\em The classical groups} \cite{Weyl},  for  the algebras $\mathcal T_{k}^n(V)=\mathcal T_{X}^n(V)$ of $GL(V)$ equivariant polynomial maps $F:End(V)^{  k}  \to End(V)^{  \otimes n},$ from $k$--tuples  $X=\{x_1,\ldots,x_k\}$ of matrix variables $End(V)^{ k}  $  to tensor spaces $End(V)^{  \otimes n}=M_d^{  \otimes n}$. \smallskip

From an algebraic point of view,  in a given basis,  so that $End(V)=M_d(\mathbb Q)$ is the algebra of $d\times d$ matrices,  the polynomial functions on  $M_d(\mathbb Q)^k$ are the polynomials $\mathbb Q[\xi^{(i)}_{a,b }]$ in the $kd^2$  variables $\xi^{(i)}_{a,b},\ i=1,\ldots,k;\ a,b=1,\ldots,d$ which are thought of as the entries of the $k$ {\em generic $d\times d$ matrices}  
 \begin{equation}\label{geme}
\Xi=\{\xi_1,\ldots,\xi_k\},\ \xi_i=(\xi^{(i)}_{a,b})\in M_d(\mathbb Q[\xi^{(i)}_{a,b }]) . 
\end{equation}  On $M_d(\mathbb Q)^k$ acts the group $GL(d,\mathbb Q)$  by simultaneous conjugation inducing  an action on $\mathbb Q[\xi^{(i)}_{a,b }]$.
The polynomial maps from $M_d(\mathbb Q)^k$ to $M_d(\mathbb Q)^{\otimes n}$ form the matrix algebra  
 \begin{equation}\label{imme}
M_d(\mathbb Q)^{\otimes n}\otimes \mathbb Q[\xi^{(i)}_{a,b }]\simeq M_{d^n}[(\mathbb Q[\xi^{(i)}_{a,b }]).
\end{equation} On this algebra we have the diagonal action of  $GL(d,\mathbb Q)$ and finally one has
 \begin{equation}\label{defeq}
\mathcal T_{k}^n(V)=\left[M_d(\mathbb Q)^{\otimes n}\otimes \mathbb Q[\xi^{(i)}_{a,b }] \right]^{GL(d,\mathbb Q)}\subset     M_d(\mathbb Q)^{\otimes n}\otimes \mathbb Q[\xi^{(i)}_{a,b }]      .
\end{equation}
 The FFT, Theorem \ref{FFT}  is fairly simple, based on  the basic Formula \eqref{foft}, the {\em interpretation Formula,} from which follows that such maps are the evaluations in  matrices of the symbolic twisted group algebra $T\langle X\rangle^{\otimes n}\ltimes \mathbb Q[S_n]$, Definition \ref{pua}. 
 
 Where $T\langle X\rangle$ is the free algebra with trace, Definition \ref{frat},
   in the variables $X=\{x_1,\ldots,x_k\}$ and $S_n$  commutes with $T\langle X\rangle^{\otimes n}$ by exchanging the tensor factors. In other words we have a surjective map  \begin{equation}\label{ilpin}
\pi_n: T\langle X\rangle^{\otimes n}\ltimes \mathbb Q[S_n]\to \mathcal T_{X}^n(V)=\left[M_d(\mathbb Q)^{\otimes n}\otimes \mathbb Q[\xi^{(i)}_{a,b }] \right]^{GL(d,\mathbb Q)}.\end{equation}  from the symbolic algebra to the algebra of equivariant maps.   

The map $\pi_n$ maps the variables $x_i$ to the generic matrices $\xi_i$ and the symmetric group $S_n$  to its copy inside  $M_d(\mathbb Q)^{\otimes n}$.  It is compatible with the trace as defined in Definition \ref{pua}.\smallskip

The SFT, Theorem \ref{SFT1}  is the heart of the paper.     In the spirit of $T$--ideals of universal algebra, cf. \S \ref{tide}, it describes the Kernels of the maps $\pi_n$. 

By the classical method of polarization and restitution one reduces to study the restriction to the multilinear elements  of  $T\langle X\rangle^{\otimes n}\ltimes \mathbb Q[S_n]$ of degree $k$ for all $k,n$. Here the main tool is the fact that these elements can be encoded by elements of $\mathbb Q[S_{n+k}]$  by a map  called {\em interpretation}  \ref{inyx}. 

 In \S  \ref{formule} we develop the basic formulas which this map satisfies and which are necessary for the proof of the main Theorem  \ref{SFT}.\smallskip

One defines in  Formula   \eqref{mira}, for each $d$,   the  {\em $d+2$ tensor Cayley Hamilton identities} $\mathfrak C_{k,d}(x),\ k=0,\ldots,d+1$, homogeneous of degree $k$ in the $x$ and for $d+1-k$ tensor
valued polynomials.
One  first deduces all identities from these ones, Theorem \ref{SFT}.   For $k=0,d,d+1$ these have classical interpretation as respectively the antisymmetrizer on $d+1$ elements, the Cayley--Hamilton identity  and the expression of $tr(x^{d+1})$ in term of $tr(x^i), \ i=1,\ldots,d$. 

For the other $k$ they are new identities. For instance $ \mathfrak C_{1,3}(x ),\  \mathfrak C_{2,3}(x )$ are
$${\footnotesize\boxed{[(1,2,3)+(1,3,2)-(1,2)- (1,3)- (2,3)+1](x\otimes 1\otimes 1+ 1\otimes x\otimes 1+ 1\otimes 1\otimes x-tr(x))} }$$ 
 $$   \boxed{(1-(1,2))\left(  [x^2\otimes 1+  1\otimes x^2+   x\otimes x] - \,tr(x)  [x \otimes 1+ 1\otimes x ] + \det(x)\right).}$$
 Finally  using a further operation $\mathtt t:T\langle X\rangle^{\otimes n}\ltimes \mathbb Q[S_n]\to T\langle X\rangle^{\otimes n-1}\ltimes \mathbb Q[S_{n-1}]$, a {\em formal partial trace}, Proposition \ref{passs1}, corresponding to  the natural trace       contraction $M_d^{\otimes n}\to   M_d^{\otimes n-1}$ one shows that
 $$\text{Formula \eqref{eFine}}\quad \mathtt   t(\mathfrak C_{k,d}(x) )=0,\ \mathtt t(\mathfrak C_{k,d}(x)\cdot 1^{n-1}\otimes x )=-(k+1) \cdot  \mathfrak C_{k+1,d}(x)$$
\begin{theorem*}[\ref{SFT1}] All   relations for equivariant maps, i.e. the elements of the kernels of the maps  $\pi_n: T\langle X\rangle^{\otimes n}\ltimes \mathbb Q[S_n]\to \mathcal T_{X}^n(V)$   for all $n$, can be {\em formally deduced}   from     the antisymmetrizer $\sum_{\sigma\in S_{d+1}}\epsilon_\sigma\sigma=\mathfrak C_{0,d}(x)$ and $tr(1)=d$.
    \end{theorem*} Recall that $\sum_{\sigma\in S_{d+1}}\epsilon_\sigma\sigma\in \mathbb Q[S_{d+1}]\subset  T\langle X\rangle^{\otimes d+1}\ltimes \mathbb Q[S_{d+1}]$.
    
    The meaning of the expression     {\em formally deduced}  is that of Universal algebra and $T$--ideals  as explained in detail in Definitions \ref{Tid} as modified in \ref{Tid1}.    This method is at the heart of the present paper.

This sheds a new light even on the very classical Cayley--Hamilton identity which is recursively constructed  from the antisymmetrizer by Formula  \eqref{eFine}.\medskip

 Before approaching  these new Theorems \ref{lmuin} and \ref{SFT}, \ref{SFT1} let us recall how the first and second fundamental theorems of invariant theory appear classically.\medskip

\subsection{The classical Theory}
The first and second fundamental theorems of invariant theory appear in H. Weyl's book {\em The classical groups} \cite{Weyl} as a basic tool to understand the representation Theory of classical groups.  One remarkable feature of these Theorems is that they are ubiquitous, they appear classically in at least 3 different and apparently unrelated forms (see the next section).

The theorems were first developed in characteristic 0  and then extended, with considerable effort, to all characteristics, see \cite{depr},\ 
 \cite{Don},\ 
 \cite{Zubkov1},\ 
 \cite{depr0}.

In this paper we present yet another form of these theorems and  will restrict to characteristic 0 (but see the very last comment).
We will also develop the theory only for the general linear group $GL(V)$, that is the group of all linear transformations of a vector space $V$   finite dimensional over a field $F$ (which one can assume to  be $\mathbb Q$).  

For the other classical groups similar results hold (one may for instance approach them as in \cite{LZ})  and will be treated elsewhere.

\subsubsection{The first   fundamental theorem}The group $GL(V)$ acts on $V$, its defining representation,  and on its dual $V^*$  by the dual action. By convention we will write {\em vectors} in $V$ with Roman letters while {\em covectors} in $V^*$ with Greek letters. We use the bracket notation; for $\phi\in V^*,\ v\in V$ we write often $\langle \phi\mid v\rangle:=\phi(v)$, so that the dual action is given by the formula
\begin{equation}\label{dac}
\langle g\phi\mid v\rangle:=\langle \phi\mid g^{-1}v\rangle,\ \forall g\in  GL(V),\ v\in V,\ \phi\in V^*.
\end{equation} In other words  $\langle g\phi\mid gv\rangle =\langle \phi\mid  v\rangle,\ \forall g\in  GL(V)$ that is the function  $\langle \phi\mid  v\rangle$ of $v$ and $\phi$ is invariant.  For this group the {\em first fundamental theorem, FFT for short,} states that the previous functions generate all invariants   of several copies of $V$ and $V^*$. That is denoting
$$(v_1,v_2,\ldots,v_h;\phi_1,\phi_2,\ldots,\phi_k)\in V^{\oplus h}\oplus (V^*)^{\oplus k}.$$
\begin{theorem}\label{fft}
The polynomial functions on $V^{\oplus h}\oplus (V^*)^{\oplus k}$  which are   $GL(V)$ invariant are generated by the $h\cdot k$ basic functions  $\langle \phi_i\mid  v_j\rangle,\ i=1,\ldots,k;$ $ j=1,\ldots,h$.
\end{theorem}

One of the remarkable features of the Theory as presented by H. Weyl is the fact that this Theorem is equivalent to  a second Theorem, with $h=k$.
\begin{theorem}\label{fft1}
The algebra of linear operators  on $V^{\otimes h} $  which commute with the diagonal action of   $GL(V)$  is generated by the  elements of the symmetric group $S_h$. The two actions are
\begin{equation}\label{twa}
g\cdot (v_1\otimes v_2\otimes \ldots\otimes v_h)= gv_1\otimes gv_2\otimes \ldots\otimes gv_h,\qquad \qquad g\in GL(V),\qquad \end{equation}\begin{equation}\label{twa1}\!\sigma\cdot(v_1\otimes v_2\otimes \ldots\otimes v_h)= v_{\sigma^{-1}(1)}\otimes v_{\sigma^{-1}(2)}\otimes \ldots\otimes v_{\sigma^{-1}(h)};\qquad \sigma\in S_h.
\end{equation}
\end{theorem}

 \subsubsection{The second  fundamental theorem}
 Together with a first fundamental theorem one has a {\em second fundamental theorem, SFT for short}. In this theorem one describes the relations among the invariants as generated by basic ones.
 
 For the setting of Theorem \ref{fft} one should think of 
 $$V^{\oplus h}=\hom(F^h,V),\  (V^*)^{\oplus k}=\hom(V,F^k) .$$ Then the invariants  $\langle \phi_i\mid  v_j\rangle$ are the entries of the  $h\times k$ matrix  image of the map
 $$\hom(F^h,V)\times \hom(V,F^k) \to \hom(F^h, F^k) ,\  (A,B)\mapsto B\circ A. $$            
 The image of this map is the variety of matrices of rank  $\leq d:=\dim V$. 
 
 The polynomial functions on $ \hom(F^h,  F^k)$ are in the variables $x_{i,j}$  which can be viewed as entries of a {\em generic matrix $X$}  and the SFT is:
 \begin{theorem}\label{sst}
The ideal of functions on the space of $h\times k$ matrices vanishing on the matrices of rank $\leq d$ is nonzero if and only if $d<\min(h,k)$. 

In this case it is generated by the determinants of all the $d+1\times d+1$  minors of the matrix $X$. 
\end{theorem} This is just the first of a long list of ideas and theorems of geometric nature on special singularities.\smallskip

Instead for the setting  of Theorem \ref{fft1} one has 
 
\begin{theorem}\label{sst1} Consider the mapping $\pi: F[S_h]\to End(V^{\otimes h })$ from the group algebra to the linear operators, given by Formula \eqref{twa1}.  

 The kernel of $\pi$  is nonzero only if $\dim V=d<h$ and then it is the two sided ideal of $F[S_h]$ generated by an antisymmetrizer $\sum_{\sigma\in S_{d+1}}\epsilon_\sigma\sigma$.

\end{theorem}Here $S_{d+1}\subset S_{h}$ and $\epsilon_\sigma=\pm 1$ is the sign of the permutation.
This is just the first of a long list of ideas and theorems in representation Theory.\smallskip

\subsection{Matrix invariants}There is still a third way in which the fundamental Theorems appear. 

In this case we take as basic representation  the direct sum of $h$ copies  of the  matrix algebra $End(V)$  of linear maps of $V$, under simultaneous conjugation by   $GL(V)$. Start with a remark.
%
%Writing $(x_1,x_2,\ldots,x_h)\in End(V)^{\otimes h}$ we have
%\begin{theorem}\label{fft2}
%The algebra  polynomial functions on $End(V)^{  h}$  which are   $GL(V)$ invariant are generated by the 
% invariants $tr(x_{i_1}x_{i_2}\ldots x_{i_k})$ for all monomials in the {\em matrix variables }  $x_i$.\end{theorem}
% This last Theorem is not as precise as the two previous ones. In general one only knows a bound on the degree of the monomials used to generate, see  \cite{razmyslov},  and not much information on minimal sets of generators.
  \begin{remark}\label{gents}
Given a vector space $W$ the symmetric group $S_h$ acts on $W^{\otimes h}$  by the formula \eqref{twa1}
%\begin{equation}\label{tess}
%\sigma \cdot  w_1\otimes\ldots\otimes w_n  =w_ {\sigma^{-1}(1)}\otimes\ldots\otimes w_ {\sigma^{-1}(n)}
%\end{equation}
thus  $S_h\subset End(W^{\otimes h})=End(W)^{\otimes h}$. Thus we have a priori two actions of $S_n$ on $End(W)^{\otimes h}$, the one given by Formula \eqref{twa1} thinking of $End(W)^{\otimes h}$ as tensors and the conjugation action  in $End(W)^{\otimes h}$  as algebra. If  $\sigma\in S_h,$ is a permutation of the tensor indices we have for a tensor  $a_1\otimes\ldots\otimes a_h\in M_d^{\otimes h}.$
 \begin{equation}\label{alco}
\sigma \circ a_1\otimes\ldots\otimes a_h \circ  \sigma^{-1}=a_ {\sigma^{-1}(1)}\otimes\ldots\otimes a_ {\sigma^{-1}(h)}=\sigma \cdot a_1\otimes\ldots\otimes a_h .
\end{equation} Formula \eqref{alco} states that these two actions coincide, cf. Formula \eqref{formuu0}.
\end{remark}

%We want to recall the basic ingredients of Theorems \ref{fft2} and  \ref{sst2} since they will be used repeatedly in what follows. 
We need to recall the multilinear
invariants of $m$ matrices, i.e. the invariant elements of the dual of  $End(V)^{\otimes m}$.
The  theorem  is, cf. \cite{kostant}: 
\begin{theorem}\label{lmuin} The space of  multilinear invariants of $m$ endomorphisms $  (x_1,x_2,\dots,x_m)$ of a  $d$--dimensional vector space $V$   is identified with the space   $End_{GL(V)}(V ^{\otimes m})$ and it will be denoted by  $\mathcal T_d(m)$. It is linearly spanned
by the functions:
\begin{equation}\label{carsv}
t_\sigma     (x_1,x_2,\dots,x_m):=tr(\sigma^{-1}\circ x_1\otimes x_2\otimes\dots\otimes   x_m),\ \sigma\in S_m.
\end{equation}
If $\sigma=(i_1i_2\dots i_h)\dots (j_1j_2\dots j_\ell)(s_1s_2\dots s_t)$ is the cycle
decomposition of $\sigma$ then  we have that
$t_\sigma     (x_1,x_2,\dots,x_m)$ equals \begin{equation}
\label{phis}=tr(x_{i_1}x_{i_2}\dots x_{i_h})\dots tr(x_{j_1}x_{j_2}\dots x_{j_\ell})
tr(x_{s_1}x_{s_2}\dots x_{s_t}).
\end{equation}
\end{theorem} 
 \begin{proof}  We recall the standard proof since this is a basic ingredient of this paper. 
 
 First remark
that  the dual of  $End(V)^{\otimes m}$ can be identified, in a $GL(V)$ equivariant way   to   $End(V)^{\otimes m}$ by
the pairing formula:
$$\langle A_1\otimes A_2\dots\otimes A_m| B_1\otimes B_2\dots\otimes B_m\rangle:=tr(A_1\otimes A_2\dots\otimes A_m\circ
 B_1\otimes B_2\dots\otimes B_m)$$$$=\prod tr(A_iB_i).$$

Under this isomorphism the multilinear invariants of matrices are identified with the $GL(V)$  invariants
of  $End(V)^{\otimes m}$ which in turn are spanned by the elements of the symmetric group, Theorem \ref{fft1}, hence by the elements of Formula \eqref{carsv}.

As for Formula \eqref{phis}, since the identity is multilinear, in the variabkes $x$, it is enough to prove it on   the
decomposable tensors of $End(V)=V\otimes V^*$ which are the endomorphisms of rank 1, i.e.  $u\otimes \phi: v\mapsto \langle\phi\,|\,v\rangle u$. We use the following basic formulas.   
\begin{lemma}\label{coss}
$$ \sigma^{-1} u_1\otimes \phi_1\otimes u_2\otimes \phi_2\otimes\ldots\otimes u_m\otimes \phi_m= u_{\sigma  (1)}\otimes \phi_1\otimes u_{\sigma  (2)}\otimes \phi_2\otimes \ldots \otimes u_{\sigma  (m)}\otimes \phi_m
  $$\begin{equation}\label{formuu0}  u_1\otimes \phi_1\otimes  \ldots\otimes u_m\otimes \phi_m\circ \sigma=  u_1\otimes \phi_{\sigma(1)}\otimes u_2\otimes \phi_{\sigma(2)}\otimes\ldots\otimes u_m\otimes \phi_{\sigma(m)}
\end{equation} 
\end{lemma}
\begin{proof}
Given $x_i:=u_i\otimes \phi_i$ and an element $\sigma\in S_m$ in the symmetric group we have
$$\sigma^{-1} u_1\otimes \phi_1\otimes u_2\otimes \phi_2\otimes\ldots\otimes u_m\otimes \phi_m(v_1\otimes v_2\otimes v_m)$$
$$= \prod_{i=1}^m\langle\phi_i\,|\,v_i\rangle u_{\sigma  (1)}\otimes u_{\sigma  (2)}\otimes \ldots \otimes u_{\sigma  (m)}
$$
$$=u_{\sigma  (1)}\otimes \phi_1\otimes u_{\sigma  (2)}\otimes \phi_2\otimes \ldots \otimes u_{\sigma  (m)}\otimes \phi_m(v_1\otimes v_2\otimes v_m)$$
similarly for the other formula.
\end{proof}
 So we need to understand in matrix formulas the invariants \begin{equation}\label{formuu}
tr(\sigma^{-1} u_1\otimes \phi_1\otimes u_2\otimes \phi_2\otimes\ldots\otimes u_m\otimes \phi_m)=\prod_{i=1}^m\langle\phi_i\,|\,u_{\sigma(i)}\rangle.
\end{equation} We need to use the rules
$$u\otimes\phi\circ v\otimes \psi=u\otimes\langle\phi\,|\,v\rangle\psi,\quad tr(u\otimes \phi)=\langle\phi\,|\,u\rangle $$ from which Formula \eqref{phis} easily follows by induction.\end{proof}
\begin{remark}\label{cong}
\begin{equation}\label{cong1}
t_{\tau\sigma\tau^{-1}}     (x_1, \dots,x_m) \stackrel{\eqref{alco}}=tr(\sigma^{-1}\circ x_{\tau(1)} \otimes\dots\otimes   x_{\tau(m)})=t_{ \sigma } (x_{\tau(1)},\dots,  x_{\tau(m) }).
\end{equation}
\end{remark}
\begin{theorem}[FFT for matrices] \label{fft2} The ring $\mathcal T_d= \mathcal T_d(X)$ of invariants of $d\times d$ matrices under simultaneous conjugation is generated by
the elements
\begin{equation}\label{trinva}
tr(x_{i_1}x_{i_2}\dots x_{i_{k-1}}x_{i_k}).
\end{equation}
\end{theorem}

 This formula means that we take all possible noncommutative monomials in the $x_i$ and form their 
traces.\begin{proof} The ring of invariants of matrices contains the ring generated by the traces of monomials and both rings are stable under polarization and restitution hence, by  Aronhold
method, see \cite{P7} Chapter 3, it is enough to prove that they coincide on multilinear elements and this is the content  of the previous Lemma.\end{proof}

 Finally for the second fundamental Theorem  for matrices one should first generalize that Theorem to a statement about the non commutative algebra of equivariant maps $F:End(V)^{  h}\to  End(V)$. This non commutative algebra    is generated over the ring of invariants by the {\em coordinates } $x_i$, see \cite{agpr}. 
 
 In this case the SFT is formulated in terms of universal algebra and the language of $T$--ideals \S \ref{tide}, see \cite{P3} and \cite{R} or Chapter 12  of \cite{agpr}.

\begin{theorem}\label{sst2} If $d=\dim_FV$, all  relations among   polynomial equivariant maps $F:End(V)^{  h}\to  End(V)$ are consequences of the Cayley--Hamilton Theorem and $tr(1)=d$.

\end{theorem} The reader may look at the proof of the second fundamental theorem \ref{sst2} as in \cite{P3}, \cite{P7}, \cite{agpr}  or \cite{depr0} since we have taken this as a model for the more general Theorem \ref{SFT} of this paper.

 This is just the first of a long list of ideas and theorems in non commutative algebra, in particular the Theory of  Cayley--Hamilton algebras as presented in the two recent preprints \cite{ppr},\ \cite{ppr1}.\smallskip

\begin{remark}\label{nos}
The first and second fundamental Theorem for matrix invariants are not as precise as the other cases. In fact from these Theorems for matrices one can only infer estimates, see \cite{R} or \S 12.2.4 of \cite{agpr},   and not a precise description of minimal generators and relations.

For $2\times 2$  matrices the results are  quite precise, due to the fact that in this case  the action of $GL(2)$  on $2\times 2$ matrices with 0 trace is equivalent to that of $SO(3)$  on its fundamental representation and then one can apply the first and second fundamental theorem for this group, see Chapter  9 of \cite{agpr}.

For $3\times 3$  matrices there are the computations in \cite{AbePit}.
\end{remark}
 
\section{Equivariant tensor polynomial maps}
\subsection{Algebras with trace\label{alwt}}
Let us quickly recall this formalism which will be used throughout this paper. 
\begin{definition}\label{awt}
An algebra  with trace is an associative algebra $A$ over some ring $F$ together with a linear map $tr:A\to A$  satisfying the following axioms, see \cite{agpr}  Chapter 2.3 or  \cite{ppr}. 
$$tr(a)\cdot b=b\cdot tr(a),\ tr(a \cdot b)=tr(b\cdot a),\ tr(tr(a)\cdot b)=tr(a)\cdot tr(b),\ \forall a,b\in A.$$ 
Then the image $tr(A)$ of $tr$  is a central subalgebra, called the {\em  trace algebra}  and $tr$ is $tr(A)$ linear.

\end{definition}
Remark that in the definition, in the approach of {\em universal algebra},  we do not assume that trace takes values in $F$.

Algebras with trace  form a category where maps are trace compatible homomorphisms. 

Algebras with trace  admit   free algebras.      \begin{definition}\label{frat}
The free algebra with trace in some list of variables $X$ is  $T\langle  X\rangle:= F\langle  X\rangle[ tr(M)]$. That is the polynomial algebra  in the elements $tr(M)$ over the usual free algebra  $F\langle  X\rangle$. 

Its trace algebra is the  polynomial algebra $F[tr(M)]$  in the elements $tr(M)$ over  $F$. By   $tr(M)$ we denote  the class of a monomial $M$ up to cyclic equivalence,      (cf.  \cite{ppr} for a detailed definition).

\end{definition}
\begin{definition}\label{faltr}
An   {\em   $n$--fold tensor  trace identity of $A$} is an element of    $T\langle  X\rangle^{\otimes n}$ vanishing for all evaluations in $A^{\otimes n}$.

In this case by   \begin{equation}\label{iltenp}
T\langle  X\rangle^{\otimes n}:=T\langle  X\rangle \otimes_{F[tr(M)]} T\langle  X\rangle\otimes  \cdots \otimes_{F[tr(M)]} T\langle  X\rangle^{\otimes n} 
\end{equation} we mean the tensor product  over the central trace  subalgebra      $F[tr(M)]$. 
\begin{remark}\label{basste}
Therefore $T\langle  X\rangle^{\otimes n}$ is a free $F[tr(M)]$ module with basis the {\em tensor monomials} $M_1\otimes M_2\otimes \ldots
\otimes M_n$. 
\end{remark}Similarly   $A^{\otimes n}$ is the tensor product over its trace algebra. Furthermore  $A^{\otimes n}$ is an algebra with trace where
$$tr(a_1\otimes \ldots \otimes a_n):=\prod_{i=1}^n tr(a_i)$$ and by evaluation we mean a homomorphism compatible with trace.

\end{definition}

\subsection{Tensor Polynomials}
 
%Given a positive integer $d$  and a field $F$ denote by $M_d(F)$ the algebra of $d\times d$ matrices with entries in $F$, or intrinsically $End(V),\dim_F V=d$.  We will in fact work with $F=\mathbb Q$ and then denote  $M_d:=M_d(\mathbb Q).$ 

This paper can be   also used as an introduction to the study  of  polynomial maps     $f: M_d^k\to M_d^{\otimes n }$  which are equivariant  under the conjugation action of $GL(\mathbb Q,d).$ Or with intrinsic notations  $f:End(V)^k\to End(V)^{\otimes n}=End(V ^{\otimes n}).  $  These maps form a non commutative algebra using the algebra structure of   $ M_d^{\otimes n }.$ Among these are the  non commutative polynomials in the tensor variables  $x_j^{(i)}:=1^{\otimes i-1}\otimes x_j  \otimes  1^{\otimes  n-i }$, where $x_j$ is a matrix variable.
\begin{definition}\label{teva}
We call the  $x_j^{(i)}$ {\em tensor variables} and the (non commutative) polynomials they generate {\em tensor polynomials}.
\end{definition}  
%The paper follows the lines of classical invariant Theory as in H. Weyl's book \cite{Weyl}. 
The concept of  {\em tensor polynomial or of tensor trace identities}   as far as I know was not considered by algebraists. I wish to thank Felix Huber for pointing out this notion which seems to play some role in Quantum Information Theory, see \cite{FH} and \cite{TDN}.

Given a free algebra $F\langle  X\rangle$ and an associative algebra $A$ over a field $F$, a map $f:X\to A$  induces a homomorphism $f:F\langle  X\rangle\to A$ and conversely a  homomorphism $f:F\langle  X\rangle\to A$ is determined by its values on $X$.

The polynomial identities of $A$ are the elements of  $F\langle  X\rangle$  vanishing under all evaluations.

Now  though,  such a map $f$  defines also, for all integers $n$,  a homomorphism of the corresponding $n$--fold tensor products:
$$f^{\otimes n}:F\langle  X\rangle^{\otimes n}\to A^{\otimes n}.$$ One can thus define as {\em $n$--fold tensor identity for $A$} an element $G\in F\langle  X\rangle^{\otimes n}$  vanishing for all evaluations in $A$, i.e. $f^{\otimes n}(G)=0,\ \forall f:X\to A$.

A similar notion holds for $A$ an algebra with trace.  
\subsection{Equivariant   maps and permutations}
We now restrict to the case  $A=M_d(F)$. Together with tensor polynomials we also have the usual invariants  of  matrices, described in Theorem \ref{fft2},  which may be viewed as  scalar valued maps to $ M_d^{\otimes n } $ i.e. multiples of the identity of  $ M_d^{\otimes n }.$

Finally  one has the {\em constant}  equivariant maps, that is the $GL(V)$  invariant elements of $ M_d^{\otimes n }.$ They form the subalgebra $\Sigma_n(V)\subset End(V)^{\otimes n}$ spanned by the permutations $\sigma\in S_n$ (Formula \eqref{twa1}) described by Theorem \ref{sst1}.\medskip

Thus we have 3 types of objects to consider, each contained in the next:
\begin{definition}\label{itret}
\begin{enumerate}\item The {\em tensor polynomial maps}, i.e.   the  polynomial  maps of  $A^{X}\to A^{\otimes n}$  induced by $F\langle  X\rangle^{\otimes n}$.\item The {\em trace tensor polynomial maps},  the maps of  $A^{X}\to A^{\otimes n}$  induced by $T\langle  X\rangle^{\otimes n}$
\item The {\em  equivariant tensor polynomial maps}, i.e.   the polynomial maps of  $A^{ X}\to A^{\otimes n}$  equivariant under conjugation by $GL(d,F)$. \end{enumerate}
\end{definition} Under the multiplication of the algebra  $ M_d(F)^{\otimes n}$  each one of these spaces forms an algebra.\smallskip

The way to understand the general form of such equivariant map, item 3., is  to associate to such a map an invariant.

Consider an equivariant polynomial  map $H(x_1,\ldots,x_k)$ of $k$, $d\times d$ matrix variables  with values in $M_d^{\otimes n}$. To this associate the  invariant scalar function of $k+n$, $d\times d$ matrix variables  $x_1,\ldots,x_k,y_1,y_2,\ldots,y_n :$
 \begin{equation}\label{pai}T(H)(x_1,\ldots,x_k,y_1,y_2,\ldots,y_n):=
tr(H(x_1,\ldots,x_k)y_1 \otimes y_2 \otimes \ldots \otimes y_n).
\end{equation} Since the trace form on  $M_d^{\otimes n}$ is non degenerate we can reconstruct $H$ from Formula    \eqref{pai}.\smallskip

By   Theorem \ref{lmuin} we have that Formula \eqref{pai} is a linear combination of products of elements of type  $tr(M)$  with $M$ a monomial  in the variables $x_1,\ldots,x_k$ and $y_1,y_2,\ldots,y_n $ and linear in these last variables.

So we say that $H$ is {\em  monomial } if it is  of the following form. There exist   monomials   $M_i, i=1,\ldots,n$  and $N_j$ in the variables $x_1,\ldots,x_k  $, possibly empty, that is with value 1, such that, setting $z_i=M_iy_i$  we have:
%$$\prod_jtr(N_j) tr(M_1y_{\tau(1)}M_2y_{\tau(2)}\ldots M_{h_1}y_{\tau({h_1})})tr(M_{h_1+1}y_{\tau({h_1+1})}M_{h_1+2}y_{\tau({h_1+2})}\ldots M_{h_1+h_2} y_{\tau({h_1+h_2})})$$$$\ldots\ldots  tr(M_{h_1+\ldots h_{k-1} }y_{\tau({h_1+\ldots h_{k-1})}}\ldots M_{h_1+\ldots h_{k } }y_{\tau({h_1+h_2+\ldots +h_{k  })}})$$  If $\sigma^{-1}$  is the product of cycles 
%$$(\tau(1) , \tau(2),\ldots, \tau( h_1 ))(\tau( h_1+1 ) ,  \tau( h_1+2 ) ,\ldots  ,\tau( h_1+h_2 )) \ldots ( \tau( h_1+\ldots h_{k-1}),\ldots  \tau( h_1+h_2+\ldots +h_{k  }) )$$
$$
T(H)=\prod_jtr(N_j) tr(z_{i_1 }z_{i_2 }\ldots z_{i_{h_1} })tr(z_{i_{h_1+1} }z_{i_{h_1+2} }\ldots z_{i_{h_1+h_2} })$$ \begin{equation}\label{monf}\ldots\ldots  tr(z_{i_{h_1+\ldots h_{k-1} }}\ldots z_{i_{n }}).
\end{equation}
If $\sigma\in S_n$ is the permutation of cycles $$\sigma=({i_1 },{i_2 },\ldots ,{i_{h_1} }) ({i_{h_1+1} },{i_{h_1+2} },\ldots, {i_{h_1+h_2} }) \ldots\ldots  ({i_{h_1+\ldots +h_{k-1} }},\ldots ,{i_{n }}) $$
then Formula \eqref{monf} becomes
 \begin{equation}\label{monf1}
T(H)=\prod_jtr(N_j) tr( \sigma^{-1} \circ  M_1 y_ 1 \otimes \dots\otimes   M_n y_{n})   ) \end{equation}
 \begin{equation}\label{monf2}
 =tr \left(\prod_jtr(N_j)  ( \sigma^{-1} \circ  M_1  \otimes \dots\otimes   M_n )    y_1 \otimes y_2 \otimes \ldots \otimes   y_n \right) \end{equation} \begin{equation}\label{monf3}
 \implies H=  \prod_jtr(N_j)   \sigma^{-1} \circ  M_1  \otimes \dots\otimes   M_n. \end{equation}
\begin{theorem}\label{FFT0}
Equivariant tensor polynomial maps are linear combinations of maps given by Formula  \eqref{monf3}.

In other words, using Definition \ref{pua}  of the next paragraph, the map $\pi_n: T\langle X\rangle^{\otimes n}\ltimes \mathbb Q[S_n]\to \mathcal T_{X}^n(V)$, from the symbolic algebra  to the algebra of equivariant maps is   surjective.
\end{theorem}This Theorem may be viewed as a First Fundamental Theorem for tensor valued equivariant functions  on matrices.

\subsubsection{The algebra $T\langle X\rangle^{\otimes n}\ltimes \mathbb Q[S_n]$} We use the definitions and notations of \S \ref{alwt} and Formula \eqref{iltenp}.

From the point of view of universal algebra the FFT says that the $n$ tensor valued    equivariant maps are the evaluations in matrices of the elements of the {\em twisted} algebra $T\langle X\rangle^{\otimes n}\ltimes \mathbb Q[S_n]$.

\begin{definition}\label{pua} For every $n\in\mathbb N, n=0,1,\cdots$ we set:\begin{enumerate}\item $T\langle X\rangle^{\otimes n}\ltimes \mathbb Q[S_n]$ is $T\langle X\rangle^{\otimes n}\otimes \mathbb Q[S_n]$ but with the commuting relations
$$  \sigma \circ M_1\otimes\ldots\otimes M_n   =M_ {\sigma^{-1}(1)}\otimes\ldots\otimes M_ {\sigma^{-1}(n)} \circ\sigma,\  \ \sigma\in S_n.$$

\item  The elements of $\mathbb Q[S_n]$ from the point of view of universal algebra are {\em constants} and are canonically evaluated by the map $\pi$ of Formula 
\eqref{twa1} in $M_d^{\otimes n}$ (cf. Theorem \ref{sst1}).

\item The algebra $T\langle X\rangle^{\otimes n}\ltimes \mathbb Q[S_n]$ is an {\em algebra with trace} (according to the   definition of page \pageref{alwt}) where  trace is defined         by Formula  \eqref{carsv}
and  
\eqref{phis}, now thought of as definitions.  

 Its trace algebra coincides with the central trace algebra $ \mathbb Q[tr(M)]$ of $T\langle X\rangle$ of Definition \ref{faltr}.\end{enumerate} 

\end{definition} 

In particular for $n=1$  we have the free algebra with trace $T\langle X\rangle$ and, for $n=0$ its trace algebra $ \mathbb Q[tr(M)]$.

 \begin{remark}\label{diffe}  The  tensor polynomial maps as algebra are identified to the algebra $F\langle  X\rangle^{\otimes n}$  modulo the vanishing elements, that is the tensor polynomial identities. Similar statement for   trace tensor polynomial maps, and $T\langle  X\rangle^{\otimes n}$, and general equivariant maps and the twisted algebra $T\langle X\rangle^{\otimes n}\ltimes \mathbb Q[S_n]$.  Observe that $F\langle  X\rangle^{\otimes n}\subset T\langle  X\rangle^{\otimes n}\subset T\langle X\rangle^{\otimes n}\ltimes \mathbb Q[S_n]$. \end{remark}
  \paragraph{Splitting the cycles} 
It order to treat the SFT it is useful to understand in a more precise way the multilinear case of the FFT. 

If the map $H$, appearing in Formula \eqref{pai},  is linear also in the variables $x_i$ then there is a    permutation $\tau\in S_{n+k}$ such that Formula  \eqref{monf} equals 
\begin{equation}\label{muls}
tr(H(x)y_1 \otimes y_2 \otimes \ldots \otimes y_n):=tr(\tau^{-1}\circ      y_1 \otimes y_2 \otimes \ldots \otimes   y_n\otimes x_1 \otimes x_2 \otimes \ldots \otimes   x_k). 
\end{equation} 
\begin{definition}\label{tkn}
We denote by   $T^{(n)}_{k,\tau}(x_1, x_2 , \ldots ,   x_k)$ the tensor valued map  $H(x)$  associated to $\tau\in S_{n+k}$ by Formula  \eqref{muls}. 

We call the map $\tau\mapsto T^{(n)}_{k,\tau}(x_1, x_2 , \ldots ,   x_k)$ the $n^{th}$--interpretation of the permutation $\tau\in  S_{n+k}$.

\end{definition}

Our next task is to describe  the elements $T^{(n)}_{k,\tau}(x_1, x_2 , \ldots ,   x_k)$  in terms of the permutation $\tau$.  We will use a simple fact on permutations that we call {\em splitting the cycles}, used in \cite{P3} to prove the SFT for matrices, Theorem \ref{sst2}.   By the {\em trivial cycles cycles} of a permutation we  mean   the ones which move elements.
 %For computational reasons we may identify $x_k=y_{n+k}$. 
   %In order to transform the previous implicit formula \eqref{muls} in one of type \eqref{monf3} we use:
   \begin{proposition}\label{spc}[Splitting the cycles]  Decompose $\{1,2,\ldots,m\}=A\cup B$ as disjoint union of two subsets $A,B$.   
 
Every permutation   $\tau\in S_m$ can be uniquely decomposed as the product  $\tau=\tau_1\tau_2\tau_3$ where:
\begin{enumerate}\item Each non trivial  cycle of $\tau_1$       contains  exactly one element of $A$.
\item $\tau_2\in S_B$   is   formed by the cycles of $ \tau $ permuting only the indices in $B$. It commutes with  $\tau_1$ and $\tau_3$ since the indices which it moves are disjoint from the ones moved by $\tau_1,\ \tau_2$.\item $\tau_3\in S_A$ permutes only the indices in $A$. \end{enumerate} 

  \end{proposition}
\begin{proof}  First we may split $\tau= \bar\tau \tau_2 $ where  $\tau_2$ collects all the cycles  of $ \tau $ permuting only the indices in $B$. Replacing $\tau$ with $\bar \tau$  we may assume that $\tau=\bar\tau_1 \bar\tau_3$  where $\bar\tau_3$ collects all the cycles of $ \tau $ permuting only the indices in $A$. 
Thus   $\bar\tau_1 $   collects all the cycles of $ \tau $ involving both the indices in $A$ and in $B$. 

By construction the 3 permutations $\bar\tau_1, \tau_2,\bar \tau_3 $ commute with each other and the indices moved by  $\tau_2$ are disjoint from the ones moved by  $\bar\tau_1$ and $\bar\tau_3$.

Then  
  the construction is based, by induction on the number of cycles, on the following identity. For $a_1,a_2,a_3\ldots,a_j$ numbers in $A$ and $ C_1,C_2,\ldots,C_k $ strings of numbers in $B$, each number appearing only once,  consider  the splitting of the  permutation cycle:
 $$\mathtt c :=(\ C_1,\ a_1,\ C_2,\  a_2,\ C_3,\ a_3,  \ \dots   ,  \ C_j,\ a_j )$$  \begin{equation}\label{sppli}
 =   ( a_1,\ C_1)(\ a_2,\ C_2)( a_3,\ C_3)   \ \dots\ (a_j,\ C_j) \circ ( a_1, a_2, a_3,\ldots,a_j)
 \end{equation}
This we call {\em splitting a cycle} with respect to $A,B$.
$$e.g.\ A=\{1,2\},\ B=\{3,4,5,6,7,8\} ;\ C_1=\ 7,8,4 ;C_2=6,3:\ a_1=1,\ a_2=2 $$$$  (1,7,8,4,2,6,3)=(\boxed{2},7,8,4 ) (\boxed{1},6,3 )\circ  (1,2) .$$
%This is based, by induction on the number of cycles, on the following identity where the indices $i\in B$ while the $j\in A$
%\begin{equation}\label{splc}
%(i_1  , i_2 ,\ldots, i_ {h_1 },j_1)(j_1,\ldots,j_k) =(i_1  , i_2 ,\ldots, i_ {h_1 },j_1 ,\ldots,j_k).\end{equation}
We then split each cycle  of  $\bar \tau_1$  obtaining for each cycle $\mathtt c_i$ of $\bar \tau_1$ a product $\mathtt c_i=d_i\circ e_i$ with $e_i$  permuting only  the indices of $A$ appearing in $\mathtt c_i$  and    each  cycle  of $d_i$ contains exactly one of the    indices of $A$ appearing in $\mathtt  c_i$. 

 Since the $\mathtt c_i$ involve moving disjoint indices   $\bar\tau_1=\prod_i \mathtt c_i=\prod_id_i\circ\prod_ie_i$.  We finally set  
 \begin{equation}\label{laspi}\tau_1:= \prod_id_i;\quad 
\tau_3:=\prod_ie_i \circ \bar\tau_3   \implies \bar\tau_1 \bar\tau_3= \tau_1\tau_3 \implies  \tau = \tau_1\tau_2\tau_3.
\end{equation}

The uniqueness is as follows. The factor $\tau_2$ is well defined from $\tau$. 

So assume we have 
$\sigma_1\sigma_3=\tau_1\tau_3$ with  $\sigma_3,\tau_3$  permuting only the indices in $A$ and  each cycle of $\sigma_1$ and of $\tau_1$        contains  exactly one element of $A$.  Multiplying both sides  by $\tau_3^{-1}$ we may assume $\tau_3=1$ so assume $\sigma_1\sigma_3 =\tau_1$.

 If $\sigma_3\neq 1$  the nontrivial  cycles of  $\sigma_3$  cannot be disjoint from those of $  \sigma_1$ otherwise they would be cycles of  $\tau_1$ a contradiction since the cycles of $\tau_1$  contain some index in $B$. 

But then there is a cycle $\mathtt c $  of   $\sigma_3$   and  some cycles $(\beta_i,i)$  of  $  \sigma_1$  with  $\beta_i$ a non trivial string in $B   $  and $i$ appears in $\mathtt c$. The product $c \prod_i (\beta,i)$ gives rise (by gluing the cycles) to a cycle    of $\tau_1$ with  more than one element in $A$,   a contradiction. \end{proof}
\begin{remark}\label{lesp}
Of course one could also split the cycle $\mathtt c$ from the {\em left} as:
 \begin{equation}\label{spplil}
 = ( a_1, a_2, a_3,\ldots,a_j)   \circ   ( a_j,\ C_1)(\ a_1,\ C_2)( a_2,\ C_3)   \ \dots\ (a_{j-1},\ C_j)
 \end{equation}\end{remark}
%Example  $A=\{1,2\},\ B=\{ 3,4,5,6,7,8\}$ then
%$$(3,5,1 ,7,6,8,2)=(3,5,1 )(7,6,8,2) (1,2) $$
%From an algorithmic point of view the two permutations   $\tau_1,\tau_2$ are obtained as follows.  Write the decomposition of $\tau$ into cycles.  Remove the product $\tau_3$ of all the cycles in $B$ and $\bar \tau_1$  the product of all the cycles in $A$, call $\bar \tau_2$ the resulting permutation.    Next 
% $\tau_2$ is obtained from $\bar \tau_2$  by {\em breaking} each cycle of $\bar \tau_2$ before each occurrence of an element $j\in A$, that is inserting $\ " )( "\,$  unless $j$ is already preceded by $"("$.    $\tau_1$ is finally the product of  $\bar \tau_1$ by the permutation obtained from $\bar\tau_2$  by simply removing all indices  $i\in B$. 
% $$e.g.\ A=\{1,2,3\},\   (\boxed{1},7,8,4,\boxed{3},6,\boxed{2},9)=(1,3,2)\circ  (1,7,8,4 ) (3,6)(2,9) .$$
 \begin{definition}\label{splp}
Denote by  $U_{A}(B)$    the set of permutations in  $S_{m }=S_{A\cup B }$ with the property that in each cycle appears at most one element of  $A$. 
\end{definition} We may say that, 
Proposition \ref{spc},\quad  [Splitting the cycles]    states that    the product map  $\pi:  U_{A}(B)\times  S_A \to S_{A\cup B },$ $  (\tau,\sigma)\mapsto \tau\circ\sigma$ is a bijection with inverse $\gamma\mapsto (\gamma_1  \gamma_2,\gamma_3)$. \begin{remark}\label{sco} 1)\quad By the uniqueness of the decomposition it follows that, 
 if $\tau\in S_A$ we have $(\sigma\circ\tau)_3=\sigma_3\circ\tau $ for all $\sigma\in S_{A\cup B }$.
 
2)\quad  Assume  $\tau=\gamma \rho$ with  $\gamma $ and $\rho$ permutations on two disjoint sets of indices, $(A_1\cup B_1), (A_2\cup B_2)$ and $A=A_1\cup A_2,\ B=B_1\cup B_2$ then, 
 \begin{equation}\label{sppr}
\tau_i=\gamma_i \rho_i,\  i=\{1,2,3\} \ \text{ for the respective decompositions.}
\end{equation}
3)\quad  If $\sigma=\alpha\circ \beta,\ \alpha \in S_A,\ \beta\in S_B$ we have:
\begin{equation}\label{conu}
(\sigma\tau\sigma^{-1})_1= \sigma\tau_1\sigma^{-1} ,\ (\sigma\tau\sigma^{-1})_2= \beta\tau_2\beta^{-1},\ (\sigma\tau\sigma^{-1})_3= \alpha\tau_3\alpha^{-1}.
\end{equation}
4)\quad $U_{A}(B)$ is stable under conjugation by elements of $S_A\times S_B$.\end{remark}
 \subsubsection{Formulas\label{formule}}
 We want to apply the previous Proposition \ref{spc} to Formula \eqref{muls}.
 
 Decompose $\{1,2,\ldots,n+k\}=A\cup B$ with $A$ the indices of type $y$ and $ B$ the ones of type $x$:$$A=\{1,2,\ldots,n\},\ B=\{n+1,   \ldots  n+k\}.$$   We identify $S_B=S_k$.
Denote, for simplicity of notations, with   the symbols 
$Y_A:=y_1 \otimes y_2 \otimes \ldots \otimes y_n ;$ and $ X_B:= x_1 \otimes x_2 \otimes \ldots \otimes   x_k $. So equation \eqref{muls} is written as
\begin{equation}\label{muls1}
tr(H(x)Y_A):=tr(\tau^{-1}\circ     Y_A\otimes X_B). 
\end{equation} If $\sigma\in S_A,\ \tau\in S_B$ we have
$$\sigma \circ  \tau \circ Y_A\otimes X_B=(\sigma \circ  Y_A)\otimes (\tau \circ    X_B )$$ Notice next that Remark \ref{cong} gives:
 
\begin{lemma}\label{dueco}
1)\quad If $\sigma\in S_B$ permutes the indices $B$  we have 
\begin{equation}\label{muls2}
T^{(n)}_{k,\sigma\tau\sigma^{-1}}(x_1 ,x_2 , \ldots , x_k)= T^{(n)}_{k,\tau}(x_{\sigma(1)}, x_{\sigma(2)}, \ldots,  x_{\sigma(k)}).
\end{equation} 
2)\quad If $\gamma\in S_A$ permutes the indices $A$  we have 
\begin{equation}\label{muls3} T^{(n)}_{k,\gamma\tau\gamma^{-1}}(x_1 ,x_2 , \ldots , x_k)=\gamma  T^{(n)}_{k,\tau}(x_1 ,x_2 , \ldots , x_k) \gamma^{-1} 
\end{equation}  is a permutation of the tensor factors.
\end{lemma} 
\begin{proof}1)\quad $$tr((\sigma\tau\sigma^{-1})^{-1}\circ     Y_A\otimes X_B)=tr(\tau^{-1}\circ     Y_A\otimes (\sigma^{-1}  \circ X_B  \circ \sigma ))$$ 
and, by Formula \eqref{alco} 
$$\sigma^{-1} \circ  X_B  \circ \sigma= x_{\sigma(1)} \otimes x_{\sigma(2)}\otimes \ldots \otimes   x_{\sigma(k)}.$$  

2)\quad      
$$tr((\gamma\tau\gamma^{-1})^{-1}\circ     Y_A\otimes X_B)=tr(\tau^{-1}\circ   (\gamma^{-1}  \circ  Y_A \circ \gamma)\otimes  X_B  )
$$
 $$\gamma^{-1}  \circ  Y_A \circ \gamma= y_{ \gamma(1)} \otimes y_{ \gamma(2)}\otimes \ldots \otimes   y_{ \gamma(n)} .$$
Thus
$$tr( T^{(n)}_{k,\gamma\tau\gamma^{-1}}(x_1 ,x_2 , \ldots , x_k)Y_A)=tr( T^{(n)}_{k, \tau }(x_1 ,x_2 , \ldots , x_k)y_{ \gamma(1)} \otimes y_{ \gamma(2)}\otimes \ldots \otimes   y_{ \gamma(n)} ) $$
%$$= tr( \gamma^{-1}T^{(n)}_{k, \tau }(x_1 ,x_2 , \ldots , x_k)\gamma\gamma^{-1}y_{ \gamma(1)} \otimes y_{ \gamma(2)}\otimes \ldots \otimes   y_{ \gamma(n)}\gamma ) $$
$$= tr( \gamma T^{(n)}_{k, \tau }(x_1 ,x_2 , \ldots , x_k)\gamma^{-1}\gamma y_{ \gamma(1)} \otimes y_{ \gamma(2)}\otimes \ldots \otimes   y_{ \gamma(n)}\gamma^{-1} ) $$$$= tr( \gamma T^{(n)}_{k, \tau }(x_1 ,x_2 , \ldots , x_k)\gamma^{-1} y_{  1 } \otimes y_{  2 }\otimes \ldots \otimes   y_{  n } ) .$$
\end{proof}

Assume  that $\tau=\gamma_1\gamma_2$ with  $\gamma_1$ and $\gamma_2$ permutations on two disjoint sets of indices, $(A_1\cup B_1), (A_2\cup B_2)$ and $A=A_1\cup A_2,\ B=B_1\cup B_2$. Then, up to  conjugating with a permutation of $S_A\times S_B$, and using Lemma \ref{dueco},  we may assume $A_1=  \{1,2,\ldots,p\}$ the first $p$ indices and similarly for $B_1=  \{n+1,n+2,\ldots,n+h\}$  so that 
$$Y_A=Y_{A_1}\otimes Y_{A_2},\quad  X_B=X_{B_1}\otimes X_{B_2}.$$
\begin{equation}\label{dett}
tr\left(( \gamma_1\gamma_2)^{-1}  \circ     Y_A\otimes X_B\right)=tr\left(  \gamma_1 ^{-1}  \circ     Y_{A_1}  \otimes X_{B_1} \right)
tr\left(   \gamma_2 ^{-1}  \circ       Y_{A_2} \otimes  X_{B_2}\right)
\end{equation}
We finally deduce 
\begin{equation}\label{ilpd}
T^{(n)}_{k,\gamma_1\gamma_2}(x_1 ,x_2 , \ldots , x_k) =T^{(p)}_{h,\gamma_1}(x_1 ,x_2 , \ldots , x_h)\otimes T^{(n-p)}_{k-h,\gamma_2}(x_{h+1} ,x_{h+2} , \ldots , x_k).
\end{equation}
Here if either $A_1$  or $A_2$ is empty, i.e. $p=0$ or $p=n$, the corresponding element is a scalar and the tensor product is to be understood as multiplication.

We use the following notation.
\begin{definition}\label{moas}
 If $C =n+ j_1,n+j_2,\ldots,n+ j_{ p}$ is a string of indices in $B$ we  say that the monomial $M  := x_{ j_1} x_{j_2}\ldots, x_{ j_{ p}}$ is   {\em  associated to $C$}.
\end{definition}

 \begin{proposition}\label{losp} If $\tau=\tau_1\tau_2\tau_3$ is the splitting relative to  the decomposition $(1,2,\ldots,n)\cup (n+1,n+2,\ldots, n+k)$ and $h$ is the number of elements of $B$ moved by $\tau_1$ we have:
\begin{equation}\label{foft}
T^{(n)}_{k,\tau}=\tau_3^{-1}\circ t_{\tau_2}(x) T^{(n)}_{h,\tau_1}(x);\    t_{\tau_2}(x)= \prod_\ell tr(N_\ell),\  T^{(n)}_{h,\tau_1}(x) =   M_{ 1 }  \otimes M_{ 2 }   \otimes \ldots \otimes  M_ { n } . 
\end{equation}
\qquad  With $N_\ell=x_{ j_1} x_{ j_2}\ldots x_{  j_{ \ell }}$ the monomials  corresponding to the cycles $(n+ j_1,n+j_2,\ldots,n+ j_{ \ell })$ of  $\tau_2$. The  $M_i=x_{ f_1} x_{ f_2}\ldots x_{ f_{ p }}$  (as in Formula \eqref{monf1}) the monomial   corresponding to the cycle  $(i, n+ f_1,n+f_2,\ldots,n+f_{ p})$ of  $\tau_1$.
\end{proposition}  
\begin{proof} Here we use the notation $T^{(n)}_{h,\tau_1}(x) $  to mean  a map in $h$  of the variables $x$ not necessarily the first.

Up to  conjugating with a permutation of $S_A\times S_B$, we may apply Remark \ref{sco},  and Formula \eqref{ilpd}. We thus reduce  to $\tau$   a single cycle $\mathtt c$. 

 If $\mathtt c$ consists only of indices in $B$ then it gives a contribution  $tr(N)$, otherwise, for some $j\geq1$, $A=\{1,2,\ldots,j\}$ and $B=\{j+1,j+2,\ldots,j+h\}$ by conjugating with a permutation of $S_A$ we may  further assume  $$\mathtt c :=(\ C_1,\ 1,\ C_2,\  2,\ C_3,\ 3,  \ \dots   ,  \ C_j,\ j )$$  \begin{equation}\label{sppli1}
 =   ( 1,\ C_1)(\ 2,\ C_2)( 3,\ C_3)   \ \dots\ (j,\ C_j) \circ ( 1, 2, 3,\ldots,j)=\mathtt c_1\circ \mathtt c_3.
 \end{equation}
 If $C_i= j+a_1,j+a_2,\ldots, j+a_{i_p}$ set $M_i:= x_{ a_1} x_{a_2}\ldots, x_{ a_{i_p}}$ its associated monomial, so
$$tr( \mathtt c^{-1} Y_j\otimes X_h)=  tr(  M_1 y_1 M_2y_ 2M_3y_3 \dots M_jy_j)  $$
$$\stackrel{\eqref{phis}}= tr( ( 1, 2, 3,\ldots,j)^{-1}  M_1y_1\otimes   M_2y_2\otimes M_3y_ 3 \otimes \dots \otimes M_j y_j) $$
$$= tr( ( 1, 2, 3\ldots,j)^{-1}   M_1 \otimes   M_2 \otimes M_3  \otimes \dots \otimes M_j   \cdot y_1\otimes    y_2\otimes  y_ 3 \otimes \dots \otimes  y_j) $$
 \begin{equation}\label{nnp}
\implies T^{(j)}_{h, \mathtt c}(X)=   \mathtt c_3^{-1}   M_1 \otimes   M_2 \otimes M_3  \otimes \dots \otimes M_j.
\end{equation} Finally
$$tr( \mathtt c_1^{-1} Y_j \otimes X_h)= \prod_{i=1}^j tr(  M_i y_i  ) =tr(M_1 \otimes   M_2 \otimes M_3  \otimes \dots \otimes M_j  \cdot y_1\otimes    y_2\otimes  y_ 3 \otimes \dots \otimes  y_j) $$
$$\implies T^{(j)}_{h, \mathtt c_1}(X)= M_1 \otimes   M_2 \otimes M_3  \otimes \dots \otimes M_j .$$
\end{proof}

\begin{theorem}\label{dect} Assume $A=\{1,2,\ldots,n\},\ B=B_1\cup B_2$ with $$B_1:=\{n+1,n+2,\ldots,n+h\},\  B_2:=\{n+h+1,n+h+2,\ldots,n+k\}.$$ \begin{enumerate}
\item     Take $\rho\in S_{A\cup B_1}$ and  $\gamma=(i,u)$   a cycle,  with $i\in A$, and $u$ a string of all the indices of $B_2$.   Set $ \ell:=\rho_3^{-1}(i)$.  If $M$ is the monomial associated to $u $ (Definition \ref{moas}), we have 
\begin{equation}\label{prrp}
\begin{matrix}T^{(n)}_{k,\gamma\rho}(x_1,\ldots,x_k)=  T^{(n)}_{h,\rho}(x_1,\ldots,x_h)\cdot 1^{\otimes i-1}\otimes M\otimes  1^{\otimes  n-i } 
\\
 T^{(n)}_{k,\rho\gamma\ }(x_1,\ldots,x_k)=   1^{\otimes \ell-1}\otimes M\otimes  1^{\otimes  n-\ell} \cdot
T^{(n)}_{h,\rho}(x_1,\ldots,x_h).\end{matrix}\end{equation}
\item  Take  $\gamma=(i,u)$   a cycle,  with $i\in B_1  $, and $u$ a string of all the indices of $B_2 $.  If $M$ is the monomial associated to $u $ (Definition \ref{moas}), we have 
\begin{equation}\label{prrp0}
\begin{matrix}
T^{(n)}_{k,\rho\gamma}(x_1,\ldots,x_k)=  T^{(n)}_{h,\rho}(x_1,\ldots,x_iM,\ldots,x_h) ,\\ T^{(n)}_{k,\gamma\rho}(x_1,\ldots,x_k)=  T^{(n)}_{h,\rho}(x_1,\ldots,Mx_i,\ldots,x_h) .
\end{matrix}
\end{equation}

\item  If $\gamma\in S_A $  \begin{equation}\label{prrp2}
\begin{matrix}
T^{(n)}_{k,\rho\gamma}(x_1 ,x_2 , \ldots , x_k)=  \gamma^{-1} T^{(n)}_{k,\rho} (x_1 ,x_2 , \ldots , x_k)  ,\\ T^{(n)}_{k,\gamma\rho}(x_1 ,x_2 , \ldots , x_k)=T^{(n)}_{k,\rho} (x_1 ,x_2 , \ldots , x_k) \gamma^{-1}.
\end{matrix}
 \end{equation}

 \item 
The inclusion $i_n:S_{n-1+k}\subset S_{n+k}$ as permutations fixing $n $ gives,  for $\tau\in S_{n +k-1}$,  that $T^{(n)}_{k,i_n(\tau)}(x_1 ,x_2 , \ldots , x_k)=T^{(n -1)}_{k,\tau}(x_1 ,x_2 , \ldots , x_k)\otimes 1.$
 \item 
The inclusion $i_k:S_{n+k-1}\subset S_{n+k}$ as permutations fixing $n+k$ gives,  for $\tau\in S_{n-1+k}$,  that $T^{(n)}_{k,i_k(\tau)}(x_1 ,x_2 , \ldots , x_k)=T^{(n )}_{k,\tau}(x_1 ,x_2 , \ldots , x_{k-1})tr(x_k).$
\end{enumerate}

\end{theorem}
\begin{proof} 1)  
Consider the splitting $\rho =\rho _1\rho _2\rho _3$ for  $A\cup B_1$.  %$$  1^{\otimes i-1}\otimes n\otimes  1^{\otimes  n-i }.T^{(n)}_{k,\rho}=1^{\otimes i-1}\otimes n\otimes  1^{\otimes  n-i }\rho _1\circ t_{\rho _3}(x) T_{h,\rho _2}(x),$$
%$$=\rho _1\circ t_{\rho _3}(x) 1^{\otimes \rho _1( i)-1}\otimes n\otimes  1^{\otimes  n-\rho _1( i) } T_{h,\rho _2}(x) $$

In order to  understand $(i,u)\rho _1$   decompose $ \rho_1=\prod_{j=1}^n (  v_j,j)$ into its cycles with the $v_j$'s strings of indices of type  $B_1$. Then    $$(i,u) \rho_1=\prod_{j\neq i} (  v_j,j) (i,u) ( v_i, i)=\prod_{j\neq i} (  v_j,j) (i,v_i,u ).$$
Thus the splitting of $\gamma\rho =(\gamma\rho _1)(\gamma\rho _2)(\gamma\rho _3)$ for  $A\cup B $ is
$$(\gamma\rho) _1=(i,u)\rho_1=\prod_{j\neq i} (  v_j,j) (i,v_i,u ),\ ( \gamma\rho) _2=\rho _2  ,\  (\gamma\rho) _3= \rho _3.$$ The cycle $(i,v_i)$  has been replaced by he cycle $(i,v_i,u)$. Thus, if $M_i$ denotes the monomial associated to the string $v_i$,   we have that the monomial associated to the string $v_i,u$ is $M_iM$. 
Formula  \eqref{nnp}   gives  the first part of Formula  \eqref{prrp}. For  the second  part  we use  the fact that  $ \rho_3(i,u)=(u,\rho_3^{-1}(i) )\rho_3=(u,h)\rho_3$ commutes with $\rho_2$ so $\rho(i,u)= \rho_1(u,h)\rho_2 \rho_3=\prod_{j\neq h} (  v_j,j) (h,u ,v_h)\rho_2 \rho_3  ,$ so $M_h$ is replaced by $MM_h$.
%
%No ma forse   $(i,a,u)$  con $a$ un indice.
%

$$e.g.\  A=\{1,2,3\}, \quad B=\{4,5,6\}\cup\{7\},\  \rho=(1,2)(2,5)(3,4,6),$$$$  T^{(3)}_{3,\rho}=(1,2) \cdot  x_2\otimes 1\otimes x_1x_3$$
$$\gamma=( 7,1),   \gamma \rho= (1,2)(2,5,7)(3,4,6),\quad  T^{(3)}_{4, \sigma}=(1,2) \cdot x_2x_4\otimes 1\otimes x_1x_3.$$
 2) This follows again from the fact that by multiplying $\rho$ by $\gamma$, in $\gamma\rho$  we modify only the cycle $(a,i)$ containing $i$  by replacing $i$ with the string $u,i$, obtaining  $(a,u, i)$ while in $\rho\gamma$ we replace $i$ with the string $i,u$.  
 
 3) The first formula follows from Remark  \ref{sco} 1). As for the second  one can use the left splitting or conjugation by $\gamma$; and 4), 5)  are clear.\end{proof}

% \begin{equation}\label{foft}
%  T_{k,\tau}=\tau_2\circ  T_{k,\tau_1}=\tau_2\circ  t_{ \tau_1'}T_{h,\tau_1"} =\tau_2\circ  \prod_jtr(N_j)   M_{ 1 }  \otimes n_{ 2 }   \otimes \ldots \otimes  M_ { n } .
%\end{equation} 
\begin{theorem}\label{FFT} The space $\mathcal T_d^{\otimes n}(k)$ of  multilinear $GL(V)$-equivariant maps  of $k$ endomorphism $  (x_1,x_2,\dots,x_k)$ of a  $d$--dimensional vector space $V$   to $End(V) ^{\otimes n}$ is identified with  $End_{GL(V)}(V ^{\otimes n+k})$.  It is linearly spanned
by the  elements $T^{(n)}_{k,\tau},\ \tau\in S_{n+k}$ of Formula \eqref{foft}.\end{theorem}
%\begin{definition}\label{inter}  
%\end{definition} 
 For instance, $n=2,\ k=1$:  
$$tr(x_1y_2y_1)=tr((1,2)\circ y_1\otimes  x_1y_2)\implies T^{(2)}_{1,(2,1,3)}(x_1)=(1,2)\circ 1\otimes x_1.$$
$$n=k=3,\quad tr(x_3x_1y_2y_1x_2y_3)=tr((1,2,3)\circ  y_1\otimes  x_3x_1y_2\otimes  x_2y_3)\implies$$$$  T^{(3)}_{3,(6,4,2,1,5,3)}(x_1,x_2,x_3)=(1,2,3)\circ 1\otimes x_3x_1 \otimes x_2.$$
\begin{remark}\label{Forid}
An important remark is that all the previous Formulas developed  using matrices still hold a the symbolic level. 

In fact for any given degrees $k,n$  we know that if $d\geq n+k$ the space $End_{GL(V)}(V ^{\otimes n+k})$ is isomorphic to $ \mathbb Q[S_{n+k}]$, i.e. there are no formal identities for $d\times d$ matrices  of degree $k$ in $n^{th}$ tensors.
\end{remark}
As a consequence we have that
\begin{remark}\label{simin} The elements  $T^{(n)}_{k,\tau}(x_1 ,x_2 , \ldots , x_k)$  should also be considered symbolically as elements of the algebra $T\langle X\rangle^{\otimes n}\ltimes \mathbb Q[S_n]$.  

They form the space $ \mathcal T_{mult} (k,n)$ of multilinear elements of degree $k$  of the algebra $T\langle X\rangle^{\otimes n}\ltimes \mathbb Q[S_n]$. 

\end{remark}
\begin{definition}\label{inyx}  
Given $\sum_{\tau\in S_{n+k}}a_\tau\tau$  the symbolic map  
\begin{equation}\label{interp}
T^{(n)}_k(\sum_{\tau\in S_{n+k}}a_\tau\tau):=\sum_{\tau\in S_{n+k}}a_\tau T^{(n)}_{k,\tau}(x_1 ,x_2 , \ldots , x_k)\in T\langle X\rangle^{\otimes n}\ltimes \mathbb Q[S_n]
\end{equation} is called the {\em $n$--interpretation} of  $\sum_{\tau\in S_{n+k}}a_\tau\tau$. It is a linear is isomorphism between $ \mathbb Q[S_{n+k}]$        and $ \mathcal T_{mult} (k,n)$.
\end{definition}
In particular the explicit Formula for $
T^{(n)}_{k,\tau}$ in Formula \eqref{foft}, is the {\em $n$ interpretation of the permutation $\tau$.}\smallskip

By the classical method of polarization and restitution  one has that Formula  \eqref{foft} describes a general, not necessarily multilinear $GL(V)$-equivariant map.\medskip

For $n=0$  this is the classical theorem of generation of invariants of matrices. For $n=1$  the classical theorem of generation of  equivariant maps  from matrices to   matrices. For $k=0$ on the other hand  it is also the classical theorem that the span of the symmetric group is the centralizer of the linear group $G=GL(d,\mathbb Q)$ on $n^{th}$ tensor space, $V^{\otimes n} $ i.e.:
$$(M_d^{\otimes n})^G=End_G(V^{\otimes n} )=\Sigma_n(V)=\pi(\mathbb Q[S_n]).$$
\section{The Second Fundamental Theorem}
\subsection{The $d+2$ basic relations}
Now, together with the  First Fundamental Theorem  we have the  Second Fundamental Theorem giving the relations among the equivariant tensor valued  maps.  

In this case we want to describe, for given $d$  and each $n $, the  elements of the   twisted  algebras $T\langle X\rangle^{\otimes n}\ltimes \mathbb Q[S_n]$ vanishing under all  evaluations  $T\langle X\rangle^{\otimes n}\ltimes \mathbb Q[S_n]\to M_d^{\otimes n}$  induced by evaluations $X\to M_d$ in $d\times d$ matrices, see Definition  \ref{pua}.

Again by the classical method of polarization and restitution one reduces to study  the multilinear relations, that is the kernel of the interpretation map  of Formula \eqref{interp}.\smallskip

One starts from Theorem \ref{sst1} that is $\sum_{\sigma\in S_{d+1}}\epsilon_\sigma =0$  as operator on $V^{\otimes d+1}$ or $\bigwedge^{d+1}V=\{0\}$.  Hence the basic identity, for $d\times d$ matrices:
\begin{equation}\label{basi}
 tr(\sum_{\sigma\in S_{d+1}}\epsilon_\sigma\sigma\circ z_1\otimes z_2\otimes z_3 \otimes \ldots \otimes z_{d+1})=0.
\end{equation} If we use, for every $0\leq k\leq d+1$,  for the variables $z_1,z_2, z_3 , \ldots ,z_{d+1}$   the
 $d\times d$ matrix variables  $ y_1,y_2,\ldots,y_{d+1-k},\ x_1,\ldots,x_k, $ Formula \eqref{pai}  produces from Formula \eqref{basi}, $d+2$  relations   \begin{equation}\label{ledr}
F_{k,d}(x_1,\ldots,x_k):=(-1)^k\sum_{\sigma\in S_{d+1}}\epsilon_\sigma T^{(d+1-k)}_{k,\sigma}(x_1,\ldots,x_k)=0;\ k=0,1,\ldots,d+1.
\end{equation}  This is a multilinear  relation of degree $k$ in the variables $x_1,\ldots,x_k$, for $d+1-k$ tensor valued equivariant maps:
 \begin{equation}\label{pai1}
tr(F_{k,d}(x_1,\ldots,x_k)y_1 \otimes y_2 \otimes \ldots \otimes y_{d+1-k})=0.
\end{equation} 
For $k=0,d,d+1$ these relations have   classical interpretations. 

For $k=d$ this is the polarized form of the $d$--Cayley--Hamilton identity
$$x^d+\sum_{i=1}^d(-1)^i\sigma_{i}(x)x^{d-i}$$   and, for $k=d+1$, it is   the polarized expression of Formula \eqref{estd}, which expresses $tr(x^{d+1})$ in terms of $tr(x^i),\ i=1,2,\ldots,d$.  In other words  the expression of the $d+1$ Newton symmetric function $\psi_{d+1}(t_1,\ldots,t_d)=\sum_{i=1}^dt_i^{d+1}$ in term of the Newton symmetric function $\psi_i(t_1,\ldots,t_d),\ i=1,2,\ldots,d$.
 \begin{equation}\label{estd}
tr(x^{d+1})+\sum_{i=1}^d(-1)^i\sigma_{i}(x)tr(x^{d-i+1}) .
\end{equation}

  In both cases   this is due to the symmetry  of formula \ref{basi} with respect to permuting the $z_i$.  
  
  For $k=0$ it is the relation   $\sum_{\sigma\in S_{d+1}}\epsilon_\sigma\sigma =0$ as operator      on $V^{\otimes d+1}$.  \medskip  

For intermediate $2\leq k\leq d$, this is still symmetric in the variables $x_1,\ldots,x_k$ so it can still be viewed as  the polarized form of a tensor identity $\mathfrak C_{k,d}(x )$, in one variable $x$, for maps to $d+1-k $ tensors, obtained by restitution: $$\mathfrak C_{k,d}(x ):=\frac {1}{k!} F_{k,d}(x ,\ldots,x ),\qquad   \text{cf. Formula \eqref{mira}}.$$  For instance, $d=2,k=1$   we have:
%$$(1,2)\circ [x\otimes 1+  1\otimes x-tr(x)]-  [x\otimes 1+ 1\otimes x-tr(x) ]= $$
$$\boxed{\mathfrak C_{1,2}(x ):=(1-(1,2))\circ [x\otimes 1+  1\otimes x-tr(x)1\otimes 1] = 0.}$$
For $d=3,k=1$ and $d=3,k=2$
{\footnotesize$$ \boxed{[(1,2,3)+(1,3,2)-(1,2)- (1,3)- (2,3)+1](x\otimes 1\otimes 1+ 1\otimes x\otimes 1+ 1\otimes 1\otimes x-tr(x)1^{\otimes 3})}$$ 
%$$2 (1,2)[x^2\otimes 1+  1\otimes x^2+   x\otimes x] -tr(x)(1,2)2[x \otimes 1+ 1\otimes x ]-2[1\otimes x^2+   x^2\otimes 1]  $$
 %$$-[tr(x^2)(1,2)+ 2x\otimes x ]  +t r(x)^2 (1,2)+2 tr(x)  [1\otimes x+  1\otimes x] +tr(x^2)1 -tr(x)^2 1$$
 $$ \mathfrak C_{2,3}(x )=(1-(1,2))\left(  [x^2\otimes 1+  1\otimes x^2+   x\otimes x] - \,tr(x)  [x \otimes 1+ 1\otimes x ] + \frac  { tr(x)^2  - tr(x^2) }2\right) $$} $$ =\boxed{(1-(1,2))\left(  [x^2\otimes 1+  1\otimes x^2+   x\otimes x] - \,tr(x)  [x \otimes 1+ 1\otimes x ] + \det(x)\right).}$$
We see a remarkable {\em factorization  theorem} through two remarkable factors. \medskip

 In order to see this in general, Theorem \ref{Fac},  let su first make a definition and recall some classical facts.

Given  two numbers $i,n\in \mathbb N,\ n>0 $  consider the  set $\mathcal P(i,n)$ of all partitions $\underline h\vdash i $ of   the form $h_1\geq h_2\geq \ldots \geq h_{n}\geq 0$. For $\underline h\in P(i,n)$  let  $T_{\underline h}(x)$ be the symmetrization of  $x^{h_1}\otimes x^{h_2}\otimes  \ldots\otimes x^{ h_{n}}$ as tensor.  E.g.  
$$\underline h=0,0,0,0;\quad  T_{\underline h}=1\otimes 1\otimes 1\otimes 1.$$
$$\underline h=2,2,0;\quad  T_{\underline h}=x^2\otimes x^2\otimes 1+x^2\otimes 1\otimes x^2+1\otimes x^2\otimes x^2,\quad T_{1.1.1}= x \otimes x \otimes  x.           $$
$$\underline h=2,1,0;\quad  T_{\underline h}=x^2\otimes x \otimes 1+x^2\otimes 1\otimes x  +x \otimes x^2 \otimes 1+x \otimes 1\otimes x^2 +1\otimes x^2\otimes x+1\otimes x\otimes x^2.$$
We then define, for $i,n\in\mathbb N,\ n> 0$:
\begin{equation}\label{Tj}
\boxed{ \mathfrak T_{i,n }(x):=\sum_{\underline h\in \mathcal P(i,n)}T_{\underline h}(x)}.
\end{equation}
$$  \text{e.g.}\ \mathfrak T_{0,2}(x)=1\otimes 1   ,\ \mathfrak T_{2,2}(x)=  x^2\otimes 1+  1\otimes x^2+   x\otimes x    .$$
$$ \mathfrak T_{i,1}(x)=x^i,\ \mathfrak T_{3,3}(x)=T_{3,0,0}+T_{2,1,0}+T_{1,1,1},\ \mathfrak T_{3,2}(x)=T_{3,0}+T_{2,1}.
$$

 Denote by
$$\det(t-x)=t^d+\sum_{i=1}^d(-1)^i\sigma_{i}(x)t^{d-i},\quad \text{the characteristic polynomial of }x.$$
Recall the formulas (cf. \cite{depr0})\begin{equation}\label{sigtra}\sigma_r(x)=\sum\limits_{\substack{h_1+2h_2+\cdots +rh_r=r\\ h_1\geq0,\ldots ,h_r\geq 0}}(-1)^r\prod_{j=1}^r\frac{(-tr(x^j))^{h_j}}{h_j!j^{h_j}}\stackrel{\eqref{carsv}}=\frac {1}{r!}\sum_{\sigma\in S_r}\epsilon_\sigma t_\sigma(x).
\end{equation} 
Now given $d,\ 0\leq k\leq d+1$ set $n:=d+1-k$, decompose $$\{1,2\ldots,d+1\}=A\cup B, \quad  A=\{1,2,\ldots,n\}.$$
Recall that, Definition \ref{splp},   $U_{A}(B)\subset S_{A\cup B} $     denotes the set of permutations   with the property that in each cycle appears at most one element of  $A$. 

  Proposition \ref{spc}, [Splitting the cycles]    states that  the product map  $U_{A}(B)\times S_A \to S_{d+1 },\  (\tau,\sigma)\mapsto \tau\circ\sigma$ is a bijection.

By Theorem \ref{dect}  2) we have, if $\tau\in S_A$ and $\sigma\in S_{d+1}$,  the identity   $T^{(n)}_{k,\sigma\tau}= \tau^{-1} T^{(n)}_{k,\sigma}$ so that  Formula \eqref{ledr} becomes

%\begin{equation}\label{laffa}
%(-1)^k F_{k,d}(x_1,\ldots,x_k)= \sum_{\tau\in S_{d+1 }}\epsilon_\tau   T_{k,\tau}=\sum_{\tau\in S_{d+1 }}\epsilon_{ \tau_1} \epsilon_{ \tau_2}   \tau_2\circ  T_{k,\tau_1}\end{equation}
%\begin{equation}\label{laffa1}=\sum_{\tau\in S_{A}}\epsilon_{ \tau } \sum_{\sigma\in U_A}\epsilon_{ \sigma}   \tau \circ  T_{k,\sigma} =(\sum_{\tau\in S_{d+1-k}}\epsilon_\tau\tau) \circ \sum_{\sigma\in U_{k,d+1}}\epsilon_{ \sigma}    T_{k,\sigma} 
%\end{equation}

\begin{equation}\label{laffa}
  \begin{matrix}
F_{k,d}(x_1,\ldots,x_k)= (-1)^k \sum_{\gamma\in S_{d+1 }}\epsilon_\gamma   T^{(n)}_{k,\gamma}(x_1,\ldots,x_k)\\ \\=  (-1)^k(\sum_{\tau\in S_{A}}\epsilon_\tau\tau) \circ \sum_{\sigma\in U_{A}(B)}\epsilon_{ \sigma}    T^{(n)}_{k,\sigma}(x_1,\ldots,x_k) .
\end{matrix} 
\end{equation}

\begin{theorem}\label{Fac}  

For all $k$ with $0\leq k\leq d$ the polynomial  $F_{k,d}(x_1,\ldots,x_k)$ is the full polarization of the {\em $n$--tensor Cayley Hamilton polynomial, with ($n:=d+1-k,\ A=\{1,\ldots, n\},$ $ B=\{n+1,\ldots, n+k\}$)}:
\begin{equation}\label{mira}
\mathfrak C_{k,d}(x):=(\sum_{\tau\in S_{A}}\epsilon_\tau\tau) \circ \left[\mathfrak T_{k,n}(x)+\sum_{j=1}^k(-1)^j \sigma_{ j}(x)\mathfrak T_{k-j,n}(x)\right].
\end{equation}
\end{theorem}
\begin{proof}   
In order to prove  Formula \eqref{mira} consider, from   Formula \eqref{laffa}:   \begin{equation}\label{ilgka}
G_{k,d}(x_1,\ldots,x_k):=  \sum_{\sigma\in U_A(B)}\epsilon_{ \sigma}    T^{(n)}_{k,\sigma}(x_1,\ldots,x_k).
\end{equation}  The set  $U_A(B)$  is stable under conjugation by the group $S_B$  and such conjugation corresponds to a permutation of the variables $x_i$ in $G_{k,d}(x_1,\ldots,x_k)$.

So, since $G_{k,d}(x_1,\ldots,x_k)$ is symmetric,  it is the polarization of the element $\frac 1{k!}  G_{k,d}(x ,\ldots,x )  $  which we need to understand.\smallskip

According to its definition, or  Proposition \ref{spc}, each permutation $\sigma$  of  $U_A(B)$ is a product   $\sigma=\sigma_1\sigma_2=\sigma_2\sigma_1$ with $\sigma_2$ the product of the cycles of $\sigma$ involving only elements of $B$.

These  two permutations $\sigma_1,$     $\sigma_2$  determine a partition of $B=B_1\cup B_2$ in two subsets. With $B_2$  the indices moved by  $\sigma_1,$   and  $\sigma_2$ is a product of  cycles  in $B_1$. Then $\sigma_1$  is  a product    of exactly $n$ cycles $\mathtt c_j$ (possibly trivial) of some lengths  $h_1+1,h_2+1,\ldots, h_n+1,\ h_j\geq 0$ with $\mathtt c_j$ containing the index $j\in A$ and the remaining $h_j$ indices in  $B_2$.
\begin{equation}\label{icic}
\sigma_1=\mathtt c_1 \mathtt c_2\ldots \mathtt c_n,\quad \mathtt c_j=(i_{j,1},i_{j,2},\ldots, i_{j,h_j},j),\ j\in A,\   i_{j,1},i_{j,2},\ldots, i_{j,h_j}\in B_2.
\end{equation}

One has  $|B_1|=k-\sum_ih_i$ and further, by Formula \eqref{foft}:
$$ T^{(n)}_{k,\sigma}(X) = T^{(0)}_{ \sigma_2} (X_1)  T^{(n)}_{k,\sigma_1} (X_2)  $$ where $T^{(0)}_{ \sigma_2}(X_1)  $  is an invariant product of traces of monomials, in the $x$ variables $X_1$ indexed by $B_1$, while  $T^{(n)}_{k,\sigma_1} (X_2)  $ is of the   form  $M_1\otimes\ldots\otimes M_n$ with the $M_j$'s  monomials  in the $x$ variables $X_2$ indexed by $B_2$.  \smallskip

Given  a partition of $B=B_1\cup B_2$ denote by $U_{A,B_2}$   the set of permutations of $A\cup B_2$  which decompose in exactly $n$ cycles each containing one index $i=1,\ldots n$,  or $i\in A$.  So we have a decomposition 
$$U_A(B)=\bigcup_{B= B_1\cup B_2}S_{B_1}\times U_{A,B_2}$$ and the following expansion of Formula \eqref{ilgka} into the various    decompositions 
$B= B_1\cup B_2  $: 
\begin{equation}\label{ilgka1}
G_{k,d}(x_1,\ldots,x_k) = \sum_{B= B_1\cup B_2} \sum_{\sigma\in S_{B_1}}\epsilon_{ \sigma}    T^{(0)}_{j,\sigma} (X_1)\sum_{\tau\in U_{A,B_2} }\epsilon_{ \tau}    T^{(n)}_{k-j,\tau} (X_2).
\end{equation} Where by $(X_1)$ resp $(X_2)$ we mean the variables among the $(x_1,\ldots,x_k) $  relative to the indices of $B_1$, resp $B_2$.\smallskip
        
For a given $j$  consider the contribution to Formula \eqref{ilgka1} from all partitions $B=B_1\cup B_2$ with  $ |B_1|=j,\  B_2|=k-j.$  When we evaluate all variables $x_i$ in a single variable $x$  all the contributions  relative to the subsets  $B_1$ with the same cardinality $j$  become equal so that 
$$\text{If}\ |B_1|=j,\quad  \sum_{ \sigma\in S_{B_1}} \epsilon_\sigma    T^{(0)}_{ \sigma }(x)\stackrel{\eqref{sigtra}}=j!\sigma_j(x).$$
 Next compute $ \sum_{  \tau\in U_{A,B_2}} \epsilon_ \tau    T^{(n)}_{ k, \tau }(x) $.
 An element      $ \tau\in U_{A,B_2}$  is uniquely of the form 
 $$ \tau=\mathtt c_1 \mathtt c_2\ldots \mathtt c_n,\quad \mathtt c_a=(i_{a,1},i_{a,2},\ldots, i_{a,h_j},a),\ a=1,\ldots,n$$ with the elements $i\in B_2$ and $B_2$ has cardinality $k-j$.  
 
 Given $n$  integers $h_1,\ldots,h_n$ summing to $k-j$  we have exactly $$\prod h_i!\binom {k-j}{  h_1,\ldots,h_n} =(k-j)!$$ such permutations which have sign $(-1)^{k-j}$. When we evaluate      all variables $x_i$ in a single variable $x$  all the contributions    become equal giving 
$(k-j)!$ times the summand $(-1)^{k-j}T_{ k-j,\sigma }(x)=x^{h_1}\otimes \ldots\otimes x^{h_n}$. The sequence   $h_1,\ldots,h_n$ is obtained by reordering a partition $\underline h\in\mathcal P(k-j,n)$    so  
$$\text{If}\ |B_2|=k-j,\quad   \sum_{\tau\in U_{A,B_2} }\epsilon_{ \tau}    T^{(n)}_{k-j,\tau} (x)        \stackrel{\eqref{Tj} }=(-1)^{k-j}(k-j)! 
\mathfrak T_{k-j,n}(x).$$      Formula \eqref{ilgka1} for $G_{k,d}(x ,\ldots,x ) $  becomes    
$$
\sum_{j=0}^k \binom kj (-1)^{k-j} j!  \sigma_j(x) (k-j)! 
\mathfrak T_{k-j},n(x) =k!  (-1)^{k } \sum_{j=0}^k   (-1)^{ j}    \sigma_j(x)  
\mathfrak T_{k-j,n}(x) .$$
Substituting  in Formula \eqref{laffa} we finally have
%
%$$
%\sum_{j=0}^k  \sigma_j(x) \frac {(-1)^{k-j}} {(k-j)!} [\sum_{ \sigma\in U_{A,B_2}}    T_{ k-j,\sigma }(x)]$$
%\begin{equation}\label{finfo}=\sum_{j=0}^k  \sigma_j(x) (-1)^{k-j}  \sum_{  h_1,\ldots,h_n\mid \sum h_i=k-j}    x^{h_1}\otimes \ldots\otimes x^{h_n}\end{equation}
\begin{equation}\label{mira1}
 F_{k,d}(x)= k!(\sum_{\tau\in S_{A}}\epsilon_\tau\tau) \circ \left[\mathfrak T_{k,n}(x)+\sum_{j=1}^k(-1)^j \sigma_{ j}(x)\mathfrak T_{k-j,n}(x)\right]
\end{equation} is the desired formula.   
\end{proof}
\begin{remark}\label{mirab}
Formula \eqref{mira} or  \eqref{mira1}  applies also to $k=d+1,\ n=0$ provided we define   $\mathfrak T_{i,0}:=tr(x^i)$.
\end{remark}

   \subsection{The second fundamental theorem}\subsubsection{$T$--ideals\label{tide}}Universal algebra is  a concept first introduced by   Garrett  Birkhoff, see \cite{GB} and    P. Chon \cite{PC}, or \cite{St} for a more extensive treatment.  If one has a class of algebras  admitting free algebras $\mathcal F(X)$,  in some set of variables $X$, a $T$--ideal is an ideal  of $\mathcal F(X)$ closed under all algebra endomorphisms of $\mathcal F(X)$, which in turn are determined by {\em substitution maps $X\to \mathcal F(X)$}.  $T$--ideals appear naturally as ideals of identities of algebras in the given class of algebras. 
   
   In this paper we need a small generalization of this notion to  take care of the tensor structure, see Definition \ref{Tid} and  \ref{Tid1}.\smallskip

   We have already remarked the relationship between the  antisymmetrizer and the $d$--Cayley--Hamilton identity.
  A well known result of Razmyslov and Procesi states that, the $T$--ideal in the free  algebra with trace, of relations for $d\times d$ matrices is generated by  the $d$--Cayley--Hamilton identity (and $tr(1)=d$) (see \cite{depr0}).\smallskip

  We start from Remark \ref{diffe} stating   that the equivariant maps are the evaluations in matrices of the elements of the {\em twisted} algebra $T\langle X\rangle^{\otimes n}\ltimes \mathbb Q[S_n]$. Here by $X=\{x_1,x_2,\ldots,x_i,\ldots\}$  we indicate variables indexed by $\mathbb N$. Thus
  \begin{definition}\label{teni}
A    {\em tensor identity}  or {\em relation}
    for $d\times d$ matrices is an element of the algebra  $T\langle X\rangle^{\otimes n}\ltimes \mathbb Q[S_n]$ vanishing under all evaluations of $X$ in $d\times d$ matrices.
    
    We denote by $I_d(n)\subset T\langle X\rangle^{\otimes n}\ltimes \mathbb Q[S_n]$ this set of tensor identities.
\end{definition}  Clearly $I_d(n)$ is a two sided ideal of $ T\langle X\rangle^{\otimes n}\ltimes \mathbb Q[S_n]$ and the algebra $\mathcal T_{X}^n(V)$, of $GL(V)$ equivariant polynomial maps, equals $T\langle X\rangle^{\otimes n}\ltimes \mathbb Q[S_n]/I_d(n).$\smallskip

 Now  there are certain operations  under which tensor identities map to tensor identities.
    
    First consider the endomorphisms, as trace algebra, of $T\langle X\rangle$, which are given by {\em substitution maps}  $g:X\to T\langle X\rangle$. Such a map $g$ induces the map $g^{\otimes n}:T\langle X\rangle^{\otimes n}\to T\langle X\rangle^{\otimes n}$ which commutes with $S_n$ and hence finally induces a map, identity on $S_n$ 
    $$ g^{\otimes n }\ltimes  1:T\langle X\rangle^{\otimes n}\ltimes  \mathbb Q[S_n]\to T\langle X\rangle^{\otimes n}\ltimes \mathbb Q[S_n].$$ The ideal     $I_d(n)$ is clearly stable under these {\em substitution maps}    $ g^{\otimes n }\ltimes  1$.
    
         Next the  natural  inclusion $S_m\times S_n\subset S_{m+n}$ induces a homomorphism of algebras, in fact an inclusion:
    \begin{equation}\label{hoa}
T\langle X\rangle^{\otimes m}\ltimes  \mathbb Q[S_m]\otimes T\langle X\rangle^{\otimes n}\ltimes  \mathbb Q[S_n]\to T\langle X\rangle^{\otimes m+n}\ltimes  \mathbb Q[S_{m+n}]
\end{equation}$$A\in T\langle X\rangle^{\otimes m},\ B\in T\langle X\rangle^{\otimes n},\ \sigma\in S_m,\ \tau\in S_n ;\quad A\sigma\otimes B\tau\mapsto A\otimes B\sigma\tau$$ and we have
\begin{equation}\label{ptid}
I_d(m) \otimes T\langle X\rangle^{\otimes n}\ltimes  \mathbb Q[S_n] +T\langle X\rangle^{\otimes m}\ltimes  \mathbb Q[S_m]\otimes  I_d(n) \subset I_d(m+n).
\end{equation} Denote for simplicity  $\mathcal T (X,n):=T\langle X\rangle^{\otimes n}\ltimes  \mathbb Q[S_n]$.
  \begin{definition}\label{Tid}
A sequence $\{J(n)\}$ of ideals $J(n)\subset \mathcal T (X,n)$  will be called a {\em $T$--ideal} if
$$
g^{\otimes n }\ltimes  1(J(n))\subset J(n),\ \forall g:T\langle X\rangle\to T\langle X\rangle$$
\begin{equation}\label{dtid} J(m)\otimes \mathcal T (X,n)
+\mathcal T (X,m) \otimes J(n)\subset J(m+n),\ \forall m,n.
\end{equation}  Clearly the intersection of $T$--ideals is still a $T$--ideal, so we define.

A   $T$--ideal $\{J(n)\}$  is {\em generated } by a subset $S\subset \bigcup_n  \mathcal T (X,n)$ if it is the minimal $T$--ideal containing $S$. \end{definition}We also say that the elements of each $\{J(n)\}$ are {\em deduced} from the elements $S$. We leave to the reader to understand how the previous Formulas translate into the {\em rules  of deduction} of the  elements in the  $T$--ideal $\{J(n)\} $ from the generating set $S$.
 \subsubsection{The $T$--ideal  of tensor identities}Clearly the relations for $d\times d$ matrices  $\{I_d(n)\}$  form a $T$--ideal. 
By the classical method of polarization and restitution  one can, studying relations or $T$--ideals,   restrict to multilinear  elements  $ \mathcal T_{mult} (k,n)$. That is:
\begin{proposition}\label{mug}
If $\{J_1(n)\}$ and $\{J_2(n)\}$ are two $T$--ideals having the same multilinear elements they coincide.

\end{proposition}  
 By definition, the space  $ \mathcal T_{mult} (k,n)$ of multilinear elements of degree $k$  in  $ \mathcal T (X,n)$ is the span of the elements depending linearly only upon the first $k$ variables $x_1,\ldots,x_k$. 
We should remark that this subspace can be identified to $\mathbb Q[S_{k+n}]$ by the Formula  \eqref{foft}, through the interpretation map  $T^{(n)}_k: \tau\mapsto 
T_{k,\tau}^{(n)}$, Definition \ref{inyx}.\medskip

As for the $T$--ideal $\{I_d(n)\}$ of tensor identities for $d\times d$ matrices, 
  we start from the  $d+2$ interpretations $F_{k,d}(x_1,\ldots,x_k)$ of the antisymmetrizer as  tensor identities  for $d\times d$ matrices given by Formula \eqref{ledr}.  Equivalently, using polarization which is one of the rules of deduction,  one could start with the 1--variable relations given by Formula \eqref{mira}. We claim
   \begin{theorem}\label{SFT}  $\{I_d(n)\}$  is generated, as $T$--ideal, from the $d+2$ interpretations $\mathfrak C_{k,d}(x)$ of the antisymmetrizer and $tr(1)=d$.
   
   In other words we may say that, every relation  for equivariant tensor valued polynomials  maps from $d\times d$ matrices to tensor products of  $d\times d$ matrices  can be {\em deduced}   from the $d+2$  identities of   Formula \eqref{mira} and $tr(1)=d$. 

\end{theorem}
\begin{proof} From Proposition \ref{mug} it is enough to restrict to multilinear relations.  The proof is then presented as an algorithm.
From Remark \ref{Forid} we know that all the formulas developed in \S \ref{formule} for multilinear equivariant maps hold also at the symbolic level.

 By Theorem \eqref{SFT} one sees that the space $ \mathcal T_{mult} (k,n)\cap I_d(n)$ of multilinear relations   is 0 unless 
   $m:=k+ n\geq d+1$.
   
 In this case it is  the  image, under the mapping $\psi_k: \tau\mapsto 
T_{k,\tau}^{(n)},\tau\in S_{k+n}$ of the two sided ideal of $\mathbb Q[S_{k+n}]$ generated by the antisymmetrizer $  A_{d+1}$.  Thus it  is  formed by   linear combinations of the $k$--interpretation (Definition \ref{inyx}) in terms of tensor valued maps of the linear generators, which we write as  $\sigma\circ \tau\circ A_{d+1}\circ \tau^{-1},\ \sigma, \tau\in S_{k+n}$, of this ideal. 

 For  $m=d+1$ and each $k$ with $0\leq k\leq d+1$, setting $ n=d+1-k$ we just have, up to scale,  only the $k^{th}$ of the  $d+2$  basic relations homogeneous of degree $k$. In other words  $I_d(n)\cap \mathcal T_{mult} (k,n)=\mathbb Q\cdot F_{k,d}(x_1,\ldots,x_k)$.
 
   For fixed $m>d+1,k$ we then decompose  $\{1,2,\ldots,m\}=A\cup B$  with $B$ the last $k$ indices (the $x$  indices) and $A$ the first $n=m-k$ indices (the $y$  indices).
 
 Finally we see,   by   Lemma \ref{dueco}, that  the conjugation action by elements of $S_A\times S_B$  commutes with the interpretation. Where $S_A$ permutes the tensor factors while $S_B$  permutes the $x$ variables.\smallskip

We need thus to understand, for $m>d+1$  and $ \sigma\circ \tau\circ A_{d+1}\circ \tau^{-1}\in \mathbb Q[S_m]$, the symbolic elements  $T_k^{(n)}(\sigma\circ \tau\circ A_{d+1}\circ \tau^{-1}),\ k=0,1,\ldots,m$, given by Formula \eqref{interp},   interpretations of   $\sigma\circ \tau\circ A_{d+1}\circ \tau^{-1}$, and prove that they are deduced from the basic relations. 

%So we fix $k,m$, set $n:=m-k$  and  decompose $\{1,2,\ldots,m\}=A\cup B,\ A:=\{1,2,\ldots,n\}$.

Given any set $I$ of $d+1$ indices out of  the set $\{1,2,\ldots,m\}$  we denote by  $A_{d+1}(I)=\sum_{\sigma\in S_I}\epsilon_\sigma\sigma\in \mathbb Q[S_m] $  the antisymmetrizer in those indices.  
 
  The element $\tau   \circ A_{d+1}\circ \tau^{-1}$  is, up to sign, the antisymmetrizer on the $d+1$  elements  of $C:=\tau(1,2,\ldots,d+1)$. Denote by  $\mathfrak A_{d+1}:=A_{d+1}(C)$.     
% So assume now that $\tau=1$ and $I=\{1,2,\ldots,d+1\}$. By the antisymmetry  of  $  A_{d+1}(I) $ we may assume that the first $k$  of the indices $1,2,\ldots, d+1$ are associated to $y$ the remaining to $x$.
 
We need to understand $T_k^{(n)}(\sigma\mathfrak A_{d+1})$.

     Decompose   $\{1,2,\ldots,m\}= C\cup D  $   and split  $ \sigma=\sigma_1\sigma_2\sigma_3$ by applying Proposition \ref{spc}   to     this decomposition.
Since $\sigma_3$ is a permutation of the indices $C$  we have  $\sigma_3\mathfrak A_{d+1}=\pm \mathfrak A_{d+1}$  so we need only analyze $\sigma_1\sigma_2\mathfrak A_{d+1}.$ 

Now,  since the indices of $\sigma_2$ are disjoint from those of $\sigma_1\mathfrak A_{d+1} $, the interpretation of $\sigma_1\sigma_2\mathfrak A_{d+1} $ is, up to permuting the tensor  factors, the tensor product of the two interpretations of $\sigma_2$ and of $\sigma_1\mathfrak A_{d+1} $, Formula \eqref{ilpd}. 

Therefore the interpretation of  $\sigma_1\sigma_2\mathfrak A_{d+1} $ is  deduced from that of $ \sigma_1\mathfrak A_{d+1} $ and we may assume we are in this case from start, denoting $\sigma_1=\phi$.

We  are   left to understand   the interpretation of $\phi\mathfrak A_{d+1} $ where $\phi=\prod_i\mathtt c^i$ is a product of its cycles $\mathtt c^i$ each containing exactly one element of $C$.  \smallskip

Use the same notations $\{1,2,\ldots,m\}= C\cup D =A\cup B,\ k=|B|,\ n=|A| $ for this relation  $T_k^{(n)}(\phi\mathfrak A_{d+1})$ (and $m=k+n$). %, let $k:=  |E_2|$
\smallskip

Assume first   $\phi=1$  and let $h:=|B\cap C|,\ p:=|A\cap C|,\ h+p=d+1$. 

Then consider a permutation $\gamma=\gamma_1\circ \gamma_2\in S_A\times S_B$ such that $$\gamma_1(A\cap C)=\{1,2,\ldots,n\},\  \gamma_2(B\cap C)=\{  n+1,n+2,\ldots,n+h\},$$ 
and let $I:=\{1,2,\ldots,n\}\cup \{  n+1,n+2,\ldots,n+h\}$. 

Then $A_{d+1}(C)=\gamma_1^{-1}A_{d+1}(I)\gamma_1,  $  so, by   Lemma \ref{dueco}, we have  that
$$T_k^{(n)}(   A_{d+1}(C))(x_1,\ldots,x_k)= \gamma_1^{-1}T_k^{(n)}(   A_{d+1}(I))(x_{\gamma_2(1)},\ldots,x_{\gamma_2(k)}) \gamma_1$$ is deduced from $T_k^{(n)}(   A_{d+1}(I))(x_1,\ldots,x_k)$. Again by Theorem \ref{dect}  we have
 \begin{equation}\label{fond}
T_k^{(n)}(   A_{d+1}(I))(x_1,\ldots,x_k)=(-1)^h\prod_{j=h+1}^k tr(x_j)  F_{k,d}(x_1,\ldots,x_h)\otimes 1^{ n-p} 
\end{equation} is deduced from the basic relation $F_{k,d}(x_1,\ldots,x_h)=T_h^{(p)}(   A_{d+1}(I))(x_1,\ldots,x_h)$.\smallskip

For a general $\phi $ let us  denote by $E$ the set of indices appearing (that is moved by)   in $\phi$  and decompose  $E=E_1\cup E_2$; with $E_1=E\cap A$ the set of indices in $E $   of type $y$ and $E_2=E\cap B$  formed by indices of type $x$.  

Next split $\phi=\phi_1\phi_2\phi_3$,  as in  Proposition \ref{spc}, with respect to this decomposition of $E$. Recall that, by construction,      $\phi=\sigma_1=\prod_j\mathtt c^{(j)}$ is a product of its cycles $\mathtt c^{(j)}$ each containing exactly one element of $C$. 

Thus we split each cycle  $\mathtt c^{(j)}= \mathtt c^{(j)}_1\mathtt c^{(j)}_2\mathtt c^{(j)}_3$  and   for each $i=1,2,3$  we have that $\phi_i=\prod_j \mathtt c^{(j)}_i$ (cf. Formula \eqref{sppr}).  

Recall that, if    $\mathtt c^{(j)}$ is formed only of elements of $A$  we have    $\mathtt c^{(j)}= \mathtt c^{(j)}_3$; let us call this set of indices $S_3$. 

If  $\mathtt c^{(j)}$  is formed entirely of elements of $B$ then   $\mathtt c^{(j)}= \mathtt c^{(j)}_2$;  let us call this set of indices $S_2$.

Otherwise   the splitting of the cycle, Formula \eqref{sppli},  is 
$\mathtt c^{(j)}=\mathtt c^{(j)}_1 \mathtt c^{(j)}_3$. 
All indices of $A$ appearing in $\mathtt c^{(j)}$  form the cycle  $ \mathtt c^{(j)}_3$  while  each of these indices appears  in one and only one of the cycles of $\mathtt c^{(j)}_1$;  let us call this set of indices $S_1$.\smallskip

     Since  $\phi_3= \prod_j \mathtt c^{(j)}_3            $  is a permutation of indices of type $y$  by Theorem \ref{dect} 3.  Formula \eqref{prrp2} we have, setting $$\overline C:=\phi_3(C),\     \overline D:=\phi_3(D),\              \overline{\mathfrak  A}_{d+1} =A _{d+1}(\phi_3(C))=A _{d+1}(\overline C )$$ that   $$T_k^{(n)}(\phi_1\phi_2\phi_3\mathfrak  A_{d+1} )= T_k^{(n)}(\phi_1\phi_2\phi_3\mathfrak  A_{d+1}\phi_3^{-1}\phi_3 )=   \phi_3^{-1}T_k^{(n)}(\phi_1\phi_2\overline{\mathfrak  A}_{d+1}).$$  
We are thus reduced to study $$T_k^{(n)}(\phi_1\phi_2\overline{\mathfrak  A}_{d+1}) =T_k^{(n)}(\prod_{j\in S_1}\mathtt c^{(j)}_1\prod_{j\in S_2}\mathtt c^{(j)} \overline{\mathfrak  A}_{d+1})  .$$   
Since $\overline C:=\phi_3(C)$ and $\phi_3\in S_A$  we have  $\overline C\cap B=C\cap B$ so    $B\cap D=B\cap \overline D$ is disjoint from  $\overline C$. 

By assumption each  cycle $\mathtt c^{(j)}$  contains a unique element $h_j$ of $C$.   If $j\in S_2$  then $h_j\in B,$ and $  \mathtt c^{(j)}=(h_j,u_j)$ with $u_j$  a string of elements of $\overline D\cap B$. 

Thus, by Formula \eqref{prrp0},  we have that $T_k^{(n)}(\prod_{j\in S_1}\mathtt c^{(j)}_1\prod_{j\in S_2}\mathtt c^{(j)} \overline{\mathfrak  A}_{d+1})   $       is obtained from     $T_k^{(n)}(\prod_{j\in S_1}\mathtt c^{(j)}_1 \overline{\mathfrak  A}_{d+1})   $ by replacing  each variable   $x_{h_j},\ j\in S_2$   with the monomial $M_jx_{h_j}$ with $M_j$ associated to the string $u_j$. This is one of the deduction rules.

Up to permuting the variables, and renaming the values of $n,k$, we are finally reduced to  analyze  $T_k^{(n)}(\prod_{j\in S_1}\mathtt c^{(j)}_1 \overline{\mathfrak  A}_{d+1})   $. 

We have to distinguish two cases.  The first for the indices $S_1^B$  such that $h_j\in B\cap C=B\cap \overline C$  and the second for  the indices   $S_1^A$  such that $h_j\in A\cap C$.  

 If  $h_j\in A$  all the cycles decomposing  $\mathtt c^{(j)}_1$  are of the form $(a,v_a),\ a\in A$ and $v_a$ a string  in $B\cap D$.  Therefore to these elements we may apply either Formula  \eqref{prrp} if $a\in \overline C$  or   Formula  \eqref{ilpd} if $a\notin \overline C$. 
 
 We are finally reduced to  analyze  $T_k^{(n)}(\prod_{j\in S_1^B}\mathtt c^{(j)}_1 \overline{\mathfrak  A}_{d+1})   $.  

Now for $j\in S_1^B$  we first remark that, since $h_j\in B$ is the only element in $\mathtt c^{(j)} $ belonging to $C$, we have that the elements $a\in A$ appearing in $\mathtt c^{(j)} $ are also    in $D$.  Hence      $\phi_3(C)= \prod_{j\notin S_1^B} \mathtt c^{(j)}_3   (C)=\overline C         $. Thus the  elements $a\in A$ appearing in $\mathtt c^{(j)},  j\in S_1^B $ are also not in $\overline C$.  

Thus  for $j\in S_1^B$ we have that $\mathtt c^{(j)}_1 $  is again a product of  cycles $(a,v_a),$ $ a\in A,\ a\notin \overline C$ and $v_a$ a string  in $B\cap D$. These  cycles  are treated as before, and finally a cycle  $(a,u_a,h_j, v_a),\ a\in A,\ a\notin \overline C$ and $u_a,v_a$ two strings  in $B\cap D$. These cycles  correspond to some subset $\bar A$ of indices of  $A$  and we will write $h_a:=h_j$ for $a\in\bar A$. 

Setting $\rho:=\prod_{a\in\bar A} (a,u_a,h_j, v_a)$ we are  reduced to analyze $T_k^{(n)}(\rho \overline{\mathfrak  A}_{d+1})   $.

Split $$ (a,u_a,h_a,v_a) = (a,v_a)(h_a,u_a)(a,h_a)=(h_a,u_a) (a,v_a)(a,h_a).$$
Let $\gamma =\prod_a (a,h_a)=\gamma ^{-1}$ so that  $\rho \gamma =\prod_a (h_a,u_a)(a,v_a) $ and write   $$\rho \overline{\mathfrak  A}_{d+1}=\rho \gamma  (\gamma  \overline{\mathfrak  A}_{d+1} \gamma ^{-1})\gamma .$$
Then $\gamma  \overline{\mathfrak  A}_{d+1} \gamma ^{-1}= A_{d+1}(\gamma (\overline C))$ is also an antisymmetrizer on $d+1$ indices, let us denote it by  $\widetilde{\mathfrak  A}_{d+1}$. Only now the $x$ indices   $h_a$ corresponding to the $a\in\bar A$ have been replaced by the $y$  indices  $a\in\bar A$ and  $\rho \overline{\mathfrak  A}_{d+1}$ has  been replaced by $$\rho \gamma  \widetilde{\mathfrak  A}_{d+1} \gamma =\prod_j(h_a,u_a) \cdot \prod_a (a,v_a)  \widetilde{\mathfrak  A}_{d+1}\prod_a (a,h_a)     .$$   

 The   indices of  $\prod_j(h_a,u_a) $ are all $x$  indices, the indices $u_a$ are  disjoint from the indices in $\prod_a (a,v_j)\widetilde{\mathfrak  A}_{d+1}\gamma  $ therefore the   interpretation of $\rho \gamma  \widetilde{\mathfrak  A}_{d+1} \gamma  $  is obtained by Formula \eqref{prrp0},  from  the   interpretation of $\prod_a (a,v_a)  \widetilde{\mathfrak  A}_{d+1}\prod_a (a,h_a) $  by substituting each variable  $x_{h_a}$  with the monomial $Mx_{h_a}$, with $M$ associated to $ u_a $. One of the rules of deduction. 

We are  thus left  with $\prod_a(a,v_a)  \widetilde{\mathfrak  A}_{d+1}\prod_a(a,h_a) $, where $a\in   \gamma (\overline C)\cap A$  and  $v_a,h_a\notin \gamma (\overline C)$ and $v_a,h_a\in B$. Thus $\prod_a (a,v_a) $ and $\prod_a(a,h_a) $   are formed by  a product of cycles      for which we can apply 
  the two Formulas \eqref{prrp}.
  
  We  conclude that the interpretation of $\prod_a(a,v_a)  \widetilde{\mathfrak  A}_{d+1}\prod_a(a,h_a) $ is obtained from that of $\widetilde{\mathfrak  A}_{d+1}= A_{d+1}(\gamma (\overline C))  $,  by multiplying from the right and from the left by tensor products of monomials.

 Finally   the interpretation  of $\widetilde{\mathfrak  A}_{d+1}= A_{d+1}(\gamma (\overline C))= A_{d+1}(\gamma \phi_3(  C))  $ is treated by the discussion leading to Formula \eqref{fond}.

 \end{proof} Notice an interesting feature of this algorithm.  The symbolic element   $T_k^{(n)}(\sigma\circ \tau\circ A_{d+1}\circ \tau^{-1}) $,    interpretation  of   $\sigma\circ \tau\circ A_{d+1}\circ \tau^{-1}$ is  deduced from just one of the basic relations $ F_{\ell,d}(x_1,\ldots,x_h)$,. On the other hand we discover the value of $\ell$ only at the end of the algorithm.

\subsection{The final theorem}
\subsubsection{Symbolic operations on  equivariant maps}
Some operations on  equivariant maps  from matrices to tensors  can be  interpreted as operations on permutations.\smallskip

Consider the following basic operations on elements of $M_d^{\otimes n}$.
\begin{align}\label{ope}
&\sigma\in S_n,\ \sigma\cdot X_1\otimes \ldots\otimes X_n =  X_{ \sigma^{-1}(1)} \otimes \ldots\otimes X_{ \sigma^{-1}(n)}\in M_d^{\otimes n} ,\\& m: X_1\otimes X_2\otimes \ldots \otimes X_{n-1}\otimes X_n\mapsto  X_1  \otimes X_2\otimes \ldots \otimes    X_nX_{n-1}\in M_d^{\otimes n-1}\\ &\mathtt t: X_1\otimes X_2\otimes \ldots \otimes X_{n-1}\otimes X_n \mapsto tr( X_n) X_1 \otimes \ldots \otimes X_{n-1 }\in M_d^{\otimes n-1}.
\end{align} One obtains many similar operations by combining these basic ones. 
\begin{lemma}\label{connt}
\begin{equation}\label{tui}
\mathtt t((n,i) \circ  X_1\otimes X_2\otimes \ldots \otimes   X_n)=X_1\otimes \ldots \otimes X_nX_i \otimes\ldots \otimes   X_{n-1 }
\end{equation}
\end{lemma}
\begin{proof} We may assume that $X_j:=u_j\otimes \phi_j,\ =1,\ldots,n$  be $n$  decomposable endomorphisms, 
$$\mathtt t((n,i)   \circ   X_1\otimes X_2\otimes \ldots \otimes   X_n)\stackrel{\eqref{formuu0}}= \mathtt   t(u_{1}\otimes \phi_1\otimes u_{ 2  }\otimes \phi_2\otimes \ldots \otimes u_{ n  }\otimes \phi_i \otimes\ldots \otimes u_{ i }\otimes \phi_n)
  $$ 
$$= u_{1}\otimes \phi_1\otimes  u_{ 2  }\otimes \phi_2\otimes \ldots \otimes  \langle \phi_n\mid u_{i} \rangle  u_{ n  }\otimes \phi_i\otimes \ldots \otimes u_{ n-1 }\otimes \phi_ { n-1 } 
$$
$$ =X_1\otimes \ldots \otimes  X_nX_i\otimes\ldots \otimes   X_ { n-1 } $$
\end{proof}
In  particular  $m=\mathtt t\circ (n,n-1)$.
Then  remark that,   if $\sigma\in S_n$ fixes $n$,  we have $\sigma \circ \mathtt t= \mathtt t\circ \sigma$.
So consider  $S_{n-1}\subset S_n$  the permutations fixing  $n$. 

We have the coset decomposition 
$$S_n= S_{n-1}\cup\bigcup_{i=1}^{n-1} S_{n-1}(n,i).$$
From the previous Lemma we deduce, for $n\geq 2$:
\begin{proposition}\label{passs}
If $\sigma\in S_{n-1} $ then  $\mathtt t\circ \sigma=\sigma\circ\mathtt  t$    and \begin{equation}\label{puss}
\mathtt t(\sigma  \circ  X_1\otimes X_2\otimes \ldots \otimes   X_n )=\sigma \circ   tr(   X_n) X_1\otimes \ldots \otimes   X_{ n-1 }, 
\end{equation} in particular $\mathtt t(\sigma)= \sigma\cdot  tr(1)$.\smallskip

If $\sigma= \tau (n,i),\ \tau\in S_{n-1} $ then \begin{equation}\label{puss1}
\mathtt    t(\sigma \circ   X_1\otimes X_2\otimes \ldots \otimes   X_n) =\tau  \circ     X_1\otimes \ldots \otimes  X_iX_n\otimes \ldots   \otimes X_{ n-1 }  .
\end{equation}
 In particular $\mathtt t(\sigma)= \tau.$\end{proposition}\begin{proposition}\label{passs1}
 Using Formulas \eqref{puss} and \eqref{puss1}, and  $\mathtt t(1):=tr(1)$  one can define  $\mathtt t$  as a {\em formal operation}  $\mathtt t:T\langle X\rangle^{\otimes n}\ltimes \mathbb Q[S_n]\to T\langle X\rangle^{\otimes n-1}\ltimes \mathbb Q[S_{n-1}]$ extending the formal trace  $tr:T\langle X\rangle\to \mathbb Q[tr(M)]$.
 
 This is a {\em partial trace} which is linear with respect to multiplication by the scalars $ \mathbb Q[tr(M)]$ and preserves the degree in $X$.

 \end{proposition}

From formulas \eqref{puss} and \eqref{puss1} we have, setting $tr(1)=d$,  for the identities  $\mathtt t(I_d(n))\subset I_d(n-1).$

Recall that the element, with $A=\{1,2,\ldots,n\}$   $$
\mathfrak C_{k,d}(x):= ( \sum_{\tau\in S_{A}}\epsilon_\tau\tau) \circ \mathfrak U_{k,d}(x) $$\begin{equation}\label{mirag}\mathfrak U_{k,d}(x) :=  \mathfrak T_{k,n}(x)+\sum_{j=1}^k(-1)^j \sigma_{ j}(x)\mathfrak T_{k-j,n}(x).
\end{equation} of Formula \eqref{mira} is an $n$--tensor identity of degree $k$    when evaluated in $d\times d$ matrices, $d=n+k-1$.  

\begin{remark}\label{unid}
From Theorem \ref{SFT}  follows in particular that there are no identities in degree $k$ on $s<d+1-k$  tensors and furthermore, up to a scalar constant,    $\mathfrak C_{k,d}(x)$ is the unique  identity in degree $k$ on $n= d+1-k$  tensors. 
\end{remark}
%Set 
%\begin{equation}\label{eFine}
% \mathfrak C_{k,d}(x):=(-1)^k(\mathfrak T_k(x)+\sum_{j=1}^k(-1)^j \sigma_{ j}(x)\mathfrak T_{k-j}(x))
%\end{equation} so that Formula \eqref{mira}  becomes
%\begin{equation}\label{miraf}
%  (\sum_{\tau\in S_{A}}\epsilon_\tau\tau) \circ \mathfrak C_{k,d}(x)
%\end{equation}
%
%next  compute  $t(T_{\underline h}(x))$ of  the symmetrization of  $x^{h_1}\otimes x^{h_2}\otimes  \ldots\otimes x^{ h_{i}}$ as tensor.  E.g.  
%$$\underline h=2,2,0;\quad  T_{\underline h}=x^2\otimes x^2\otimes 1+x^2\otimes 1\otimes x^2+1\otimes x^2\otimes x^2$$
\begin{theorem}\label{Fine} Upon specialyzing $tr(1)=d$  we have, for $n\geq 1\iff k\leq d$:
\begin{equation}\label{eFine}
\mathtt   t(\mathfrak C_{k,d}(x) )=0,\ \mathtt t(\mathfrak C_{k,d}(x)\cdot 1^{n-1}\otimes x )=-(k+1)\cdot   \mathfrak C_{k+1,d}(x).
\end{equation}

\end{theorem}
\begin{proof} For $n=1,\ k=d$ we have $\mathfrak C_{d,d}(x)=x^d+\sum_{i=1}^d(-1)^i\sigma_{i}(x)x^{d-i}$ is the Cayley--Hamilton element and
$$\mathtt t(x^d+\sum_{i=1}^d(-1)^i\sigma_{i}(x)x^{d-i})=tr(x^d)+\sum_{i=1}^d(-1)^i\sigma_{i}(x)tr(x^{d-i})= 0,\  $$is the recursive formula expressing Newton symmetric functions in term of elementary ones. Finally $$ \mathtt t(x^{d+1}+\sum_{i=1}^d(-1)^i\sigma_{i}(x)t^{d-i+1})=tr(x^{d+1})+\sum_{i=1}^d(-1)^i\sigma_{i}(x)tr(x^{d-i+1})$$ is the Formula $\mathfrak C_{d+1,d}(x)$ expressing the $d+1$   Newton symmetric function in $d$ variables in term of the preceding ones.  

So assume $n\geq 2$,
both elements $\mathtt  t(\mathfrak C_{k,d}(x) )$ and $ \mathtt  t(\mathfrak C_{k,d}(x)\cdot 1^{n-1}\otimes x )$ are  tensor identities on $n-1=d  -k$ tensors, respectively of degree $k$ and $k+1$ for $d\times d$ matrices.  

\noindent Thus by the previous remark, on   degree of identities,   we have $\mathtt t(\mathfrak C_{k,d}(x) )=0$ and $\mathtt t(\mathfrak C_{k,d}(x)\cdot 1^{n-1}\otimes x )=\alpha\cdot  \mathfrak C_{k+1,d}(x)$ for some scalar $\alpha$.

 Observe that, by Proposition \ref{passs}
 $$  \sum_{\tau\in S_{1,2,\ldots,n}}\epsilon_\tau\tau=  (\sum_{\tau\in S_{1,2,\ldots,n-1}}\epsilon_\tau\tau )(1-\sum_{i=1}^{n-1}(i,n)) $$
 $$\implies\mathtt  t((\sum_{\tau\in S_{1,2,\ldots,n}}\epsilon_\tau\tau) \circ    \mathfrak U_{k,d}(x) \cdot 1^{n-1}\otimes x )$$ $$ \stackrel{\eqref{puss}}=(\sum_{\tau\in S_{1,2,\ldots,n-1}}\epsilon_\tau\tau )\mathtt t((1-\sum_{i=1}^{n-1}(i,n)) ) \circ    \mathfrak U_{k,d}(x) \cdot 1^{n-1}\otimes x ).$$  Thus $ \mathtt t((1-\sum_{i=1}^{n-1}(i,n))   \circ    \mathfrak U_{k,d}(x) \cdot 1^{n-1}\otimes x )=\alpha \cdot  \mathfrak U_{k+1,d}(x) $.

 We compute $\alpha$ as coefficient of the leading term $1^{\otimes n-2}\otimes x^  {k+1}$ in the previous Formula.
 
  This term arises only in $ \mathtt t((1-\sum_{i=1}^{n-1}(i,n)) ) \circ    \mathfrak T_{k,n}(x) \cdot 1^{n-1}\otimes x )$. In  fact   $ \mathfrak T_{k,n}(x) =\sum_{\underline h\in \mathcal P(k,n)}T_{\underline h}(x)  $  and we see, from    Formula \eqref{tui},  that  the only contributions can arise from $\mathtt t( - (n-1,n) A)$ with $A$ the terms 
 $$(\sum_{i+j=k,   }1^{\otimes n-2}\otimes x^  {i }\otimes  x^  {j })\cdot 1^{n-1}\otimes x =\sum_{i =0}^k1^{\otimes n-2}\otimes x^  {i }\otimes  x^  {k-i +1}$$$$\mathtt t( - (n-1,n)    \circ    \sum_{i =0}^k1^{\otimes n-2}\otimes x^  {i }\otimes  x^  {k-i +1})=-(k+1)\cdot 1^{\otimes n-2}\otimes x^  {k+1} .$$
% For  $\mathfrak C_{k+1,d}(x)$ we take $(\sum_{\tau\in S_{1,2,\ldots,n-1}}\epsilon_\tau\tau) \circ   T_{k+1,n-1}(x)$. For the element  $\mathtt t(\mathfrak C_{k,d}(x)\cdot 1^{n-1}\otimes x )$ we look at  $\mathtt t((\sum_{\tau\in S_{1,2,\ldots,n}}\epsilon_\tau\tau) \circ   T_{k ,n}(x)\cdot 1^{n-1}\otimes x )$.
% We need to compute the coefficient of the term $1^{\otimes n-2}\otimes x^  {k+1}$  of  $ T_{k+1,n-1 }(x)$ in the sum 
% $\mathtt t((1-\sum_{i=1}^{n-1}(i,n)) ) \circ   \mathfrak U_{k,d}(x) \cdot 1^{n-1}\otimes x )$.
 
% Now  $ T_{k ,n}(x)$ is the sum of the $n$  terms  $1^{\otimes i-1}\otimes x^k\otimes 1^{\otimes n-i}$ so, from  Formula \eqref{tui} we have a contribution $1^{\otimes n-2}\otimes x^  {k+1}$ only from  $$\mathtt t( - (n-1,n)    \circ    ( 1^{\otimes n-2}\otimes x^k\otimes 1)1^{n-1}\otimes x)=-1^{\otimes n-2}\otimes x^  {k+1} .$$   
   
 \end{proof}
 The specialization $tr(1)=d$ is necessary since for instance formally
 $$\mathtt t(\mathfrak C_{1,2}(x ))=\mathtt t((1-(1,2))\circ [x\otimes 1+  1\otimes x-tr(x))$$
 $$= tr(1)x+tr(x)  -tr(x) tr(1) -2x+tr(x)=( tr(1)-2)(x- tr(x)) .$$
$$\mathtt  t( \sum_{\tau\in S_{ d+1}}\epsilon_\tau\tau )\stackrel{\ref{passs}}=( tr(1)-d)\sum_{\tau\in S_{ d }}\epsilon_\tau\tau.$$
 \begin{exercise}\label{eser}
$\mathtt t(\mathfrak C_{k,d}(x ))=(tr(1)-d) \mathfrak C_{k,d-1}(x ) ,\ \forall k\leq d$.
\end{exercise}
 \begin{remark}\label{mufo}
For the multilinear identities of Formula \eqref{ledr} we have
$$F_{k+1,d}(x_1,\ldots,x_{k+1})=\mathtt  t(F_{k,d}(x_1,\ldots,x_k)\cdot 1^{\otimes d-k}\otimes x_{k+1}).$$
\end{remark}

 At this point one should introduce the operation $\mathtt  t$   in the definition of the {\em algebras} to be used to  deduce an identity from another.  So we change the definition \ref{Tid} of $T$ ideal  asking: 
 
 \begin{definition}\label{Tid1}
A sequence $\{J(n)\}$ of ideals $J(n)\subset \mathcal T (X,n)$  will be called a {\em $T$--ideal} if besides the conditions of  
 Definition \ref{Tid} it is also   stable under $\mathtt t$.   \end{definition}
 
 Under this new definition we finally have the conclusive  result.
  
   \begin{theorem}\label{SFT1}[SFT for equivariant maps]  The ideal $\{I_d(n)\}$  is generated, as $T$--ideal, by     the antisymmetrizer $\sum_{\sigma\in S_{d+1}}\epsilon_\sigma\sigma$ and $tr(1)=d$.
   
  \end{theorem}\begin{proof}
Formula \eqref{eFine}  gives recursively the $d+2$ formulas  $\mathfrak C_{k,d}(x) $  from  the antisymmetrizer $\mathfrak C_{0,d}(x) =\sum_{\sigma\in S_{d+1}}\epsilon_\sigma\sigma$ and then we apply Theorem \ref{SFT}.
\end{proof}
There is a final remarkable fact. 

Assume we take the algebras  $T\langle X\rangle^{\otimes n}\ltimes \mathbb Q[S_n]/\bar I_d(n)$  modulo the $T$--ideal $\bar I_d$ generated by the antisymmetrizer $A_{d+1}$  and no condition on $tr(1)$. From Exercise \ref{eser} we have
$$\mathtt t(A_{d+1}) =(tr(1)-d)A_{d }\implies \mathtt t^d(A_{d+1}) =\prod_{i=1}^d(tr(1)-i)\in \bar I_d.$$
The algebra $\mathbb Q[\lambda]/ \prod_{i=1}^d(\lambda-i)=\oplus_{i=1}^d \mathbb Q$ and so $T\langle X\rangle^{\otimes n}\ltimes \mathbb Q[S_n]/\bar I_d(n)$ decomposes as a direct sum of $d$ summands, in the $i^{th}$ summand we have $tr(1)=i$. But now by the same formula  $\mathtt t(A_{d+1}) =(tr(1)-d)A_{d }$ we deduce from $A_{d+1}$ and $tr(1)=i$ that in the the $i^{th}$ summand we have also $A_{i+1}=0$. Therefore we deduce the decomposition as direct sum of the  $d$ algebras of equivariant maps for $i\times i$ matrices, $i=1,\ldots,d$.
\begin{theorem}\label{suff}
$$T\langle X\rangle^{\otimes n}\ltimes \mathbb Q[S_n]/\bar I_d(n)=\oplus_{i=1}^dT\langle X\rangle^{\otimes n}\ltimes \mathbb Q[S_n]/  I_i(n)=\oplus_{i=1}^d\mathcal T_{X}^n(\mathbb Q^i).$$
\end{theorem}
\section{The algebra of equivariant maps}
\subsection{The structure of $\mathcal T_{X}^n(V)$}
What can we say about the algebra $\mathcal T_{X}^n(V)$? We assume that $X$ has at least 2 elements, the case of just one variable being special and left to the reader.  Let us first recall the Theory  for $n=1$, for a detailed study we refer to the book \cite{agpr}. 

The algebra $T_d\langle \Xi\rangle:=\mathcal T_{X}^1(V)$   is the free algebra with trace     $T\langle X\rangle$ modulo   the $d$--Cayley Hamilton identity and $tr(1)=d$. This   algebra  is a domain  generated by $k$ generic matrices  $\Xi=\{\xi_1,\ldots,\xi_k\}  $ of Formula \eqref{geme}, and the traces of their monomials. 

 Its center is the algebra of  invariants $T^{(0)}_d\langle \Xi\rangle$. If $Q^{(0)}_d\langle \Xi\rangle$ is the field of fractions of $T^{(0)}_d\langle \Xi\rangle$ then $Q _d\langle \Xi\rangle:=T_d\langle \Xi\rangle\otimes_{T^{(0)}_d\langle \Xi\rangle}Q^{(0)}_d\langle \Xi\rangle$ is a division algebra of dimension $d^2$  over its center {\em  the equivariant rational functions}. 

The subalgebra $ \mathbb Q[\xi_1,\ldots,\xi_k]\subset T_d\langle \Xi\rangle $ is called the {\em algebra of generic matrices} and it is the free algebra modulo the polynomial identities of   $d\times d$  matrices.  One of the remarkable Theorems of the theory  is that $ \mathbb Q[\xi_1,\ldots,\xi_k]$ has a non trivial center $\mathfrak Z_d(X)\subset T^{(0)}_d\langle \Xi\rangle$. 

An element $c\in \mathfrak Z_d(X)$    with no constant coefficient is called a {\em  central polynomial}.  Moreover the fields of fractions of $\mathfrak Z_d(X)$ and $T^{(0)}_d\langle \Xi\rangle$ coincide. In fact from a strong Theorem of M. Artin  \cite{Az},  \cite{ArM} 
one has that   (cf. \cite{agpr} Theorem 10.3.2), if $c $ is a central polynomial 
 \begin{equation}\label{egg}
\mathbb Q[\xi_1,\ldots,\xi_k][c^{-1}]= T_d\langle \Xi\rangle[c^{-1}] 
\end{equation}   is an Azumaya algebra of rank $d^2$ over its center $T^{(0)}_d\langle \Xi\rangle[c^{-1}]$.

Take  the tensor product $$Q _d\langle \Xi\rangle^{\otimes n} := Q _d\langle \Xi\rangle\otimes_{Q^{(0)}_d\langle \Xi\rangle} Q _d\langle \Xi\rangle \ldots\otimes  Q _d\langle \Xi\rangle\otimes_{Q^{(0)}_d\langle \Xi\rangle} Q _d\langle \Xi\rangle$$ of $n$ copies of  $Q _d\langle \Xi\rangle$ over its center $ Q^{(0)}_d\langle \Xi\rangle.$  This is  a central simple algebra  contained in the matrix algebra 
$$M_d(\mathbb Q(\xi^{(i)}_{h,k}))^{\otimes n}:=$$
$$ \footnotesize{  M_d(\mathbb Q(\xi^{(i)}_{h,k}))\otimes_{\mathbb Q(\xi^{(i)}_{h,k})} \otimes M_d(\mathbb Q(\xi^{(i)}_{h,k}))\ldots\otimes M_d(\mathbb Q(\xi^{(i)}_{h,k}))\otimes_{\mathbb Q(\xi^{(i)}_{h,k})} M_d(\mathbb Q(\xi^{(i)}_{h,k}))}$$  and 
$$ M_d(\mathbb Q(\xi^{(i)}_{h,k}))^{\otimes n}=  M_d(\mathbb Q)^{\otimes n}\otimes_{\mathbb Q  } Q(\xi^{(i)}_{h,k})$$

\begin{lemma}\label{tesra}
\begin{equation}\label{goe}
   Q _d\langle \Xi\rangle^{\otimes n} \otimes_{Q^{(0)}_d\langle \Xi\rangle}\mathbb Q(\xi^{(i)}_{h,k}))= M_d(\mathbb Q(\xi^{(i)}_{h,k}))^{\otimes n} $$$$
    \implies   Q _d\langle \Xi\rangle^{\otimes n}  = \left(M_d(\mathbb Q(\xi^{(i)}_{h,k}))^{\otimes n}\right) ^{GL(d,\mathbb Q)}.\end{equation}

\end{lemma}
\begin{proof}  The natural map of  $ Q _d\langle \Xi\rangle^{\otimes n} \otimes_{Q^{(0)}_d\langle \Xi\rangle}\mathbb Q(\xi^{(i)}_{h,k}))$ to $ M_d(\mathbb Q(\xi^{(i)}_{h,k}))^{\otimes n}$ is an isomorphism since they are both central simple algebras  of the same dimension $d^{2n}$  over  the field $\mathbb Q(\xi^{(i)}_{h,k}))$.

Since $Q^{(0)}_d\langle \Xi\rangle=\mathbb Q(\xi^{(i)}_{h,k}))^{GL(d,\mathbb Q)}$ the second claim follows.
\end{proof}  In the same way we have  the universal faithfully flat splitting, \cite{agpr} Corollary 10.4.3.
\begin{equation}\label{azi0}
\mathbb Q[\xi_1,\ldots,\xi_k][c^{-1}]  ^{\otimes n}\otimes_{T^{(0)}_d\langle \Xi\rangle[c^{-1}]} \mathbb Q[\xi^{(i)}_{h,k}][c^{-1}] =  M_d(\mathbb Q[\xi^{(i)}_{h,k}][c^{-1}])^{\otimes n}   
\end{equation}
and an isomorphism at the level of Azumaya algebra
\begin{equation}\label{azi}
\mathbb Q[\xi_1,\ldots,\xi_k][c^{-1}]  ^{\otimes n}  = \left(M_d(\mathbb Q[\xi^{(i)}_{h,k}][c^{-1}])^{\otimes n}\right) ^{GL(d,\mathbb Q)}$$$$=\left(M_d(\mathbb Q)^{\otimes n}\otimes_{ \mathbb Q}\mathbb Q[\xi^{(i)}_{h,k}][c^{-1}]\right) ^{GL(d,\mathbb Q)},
\end{equation}
 From Formula \eqref{defeq} we have that $\mathcal T_{X}^n(V)\subset Q _d\langle \Xi\rangle^{\otimes n} $ an we claim
 \begin{theorem}\label{tenaz}
$\mathcal T_{X}^n(V)$ is a prime algebra, if $|X|>1$  its center  is $T^{(0)}_d\langle \Xi\rangle$   and, for all central polynomial  $c $     we have
\begin{equation}\label{locca}
\mathcal T_{X}^n(V)[c^{-1}]=  \mathbb Q[\xi_1,\ldots,\xi_k][c^{-1}]  ^{\otimes n}
\end{equation} 
The tensor power is with respect to the center of  $ \mathbb Q[\xi_1,\ldots,\xi_k][c^{-1}] $.
\end{theorem}  If $|X|>1$ we have that two generic matrices generate $d\times d$ matrices and their corresponding tensor variables generate the tensor power which is a matrix algebra with center $\mathbb Q$ therefore the center of  $\mathcal T_{X}^n(V)$ is formed by the scalar valued equivariant maps, that is    the invariants $T^{(0)}_d\langle \Xi\rangle$.

 In order to prove the remaining part of the   Theorem let us recall a Theorem which   is attributed to Oscar Goldman in  the book of M. A. Knus, M. Ojanguren,  \textit{Th\'eorie de la descente et alg\`ebres d'Azumaya} page 112 
\cite{KnusO}.

If $R$  is a rank $n^2$ Azumaya algebra over its center $A$   the map $$\pi:R\otimes _AR^{op}\to End_A(R),\ \pi(
\sum_ia_i\otimes b_i)(x)=\sum_ia_ixb_i$$ is an isomorphism. 

Then define $\mathtt s \in R\otimes_A$ by $\pi(\mathtt s)(x)=tr(x).$ The element $\mathtt s$ is called the {\em Goldman element}.
\begin{theorem}\label{Gol}
We have
\begin{equation}\label{Gol1}
\mathtt s^2=1,\quad \mathtt s(a\otimes b)\mathtt s^{-1}=b\otimes a
\end{equation}Moreover for every map $A\to B$ so that $B\otimes_AR\simeq M_n(B)$ the element $\mathtt s$   maps to the permutation operator $(1,2)\in M_n(B)\otimes_BM_n(B)$.
\end{theorem}   
% \begin{proof}
%By definition ,  and the trace of $ M_n(B)$ restricted to $R$ takes values in $A$.
%
%
%
%Under the faithfully flat extension $A\to B$ the element $\mathtt s\mapsto  \sum_{i,j=1}^ne_{i,j}\otimes e_{j,i}$, by uniqueness since this element satisfies the same property:
%$$\sum_{i,j=1}^ne_{i,j}e_{h,k}e_{j,i} =\begin{cases}
%0\ \text{if}\ h\neq k\\
%1 \ \text{if}\ h= k
%\end{cases}.$$
%$$\mathtt s^2=\sum_{i,j=1}^ne_{i,j}\otimes e_{j,i}\sum_{h,k=1}^ne_{h,k}\otimes e_{k,h} =\sum_{i,j=1}^ne_{i,j}\otimes e_{j,i}$$
%The second  property can again be verified in the split algebra and it is
%$$\sum_{i,j=1}^ne_{i,j}\otimes e_{j,i}e_{a,b}\otimes e_{c,d}\sum_{h,k=1}^ne_{h,k}\otimes e_{k,h}=e_{c,d}\otimes  e_{a,b}. $$
%
%Finally $\mathtt s$ is the switch operator on $B^n\otimes B^n$ since
%$$\sum_{i,j=1}^ne_{i,j}\otimes e_{j,i}e_{a }\otimes e_{b} =e_{b }\otimes e_{a} .$$\end{proof}
%If  $R$  is a free rank $n^2$  module over $A$ with basis $a_1,a_2,\ldots,a_{n^2}$  then there is a unique dual basis for the trace form $tr(xy)$ that is there are unique elements  $a_1^*,a_2^*,\ldots,a_{n^2}^*$ with $tr(a_ia_j^*=\delta_i^j$.  Then we have  $$\mathtt s=\sum_{i=1}^{n^2} a_i\otimes a_i^*.$$
%This depends upon the fact that the element $\sum_{i=1}^{n^2} a_i\otimes a_i^* $ is independent of the  basis chosen so under the faithfully flat splitting it coincides with $\sum_{i,j=1}^ne_{i,j}\otimes e_{j,i}$ since these is given by the basis dual to that of elementary matrices.\smallskip

\begin{proof}
[Proof of Theorem \ref{tenaz}]We have a homomorphism $\pi_n$ of Formula \eqref{ilpin}

$$\pi_n:T\langle X\rangle^{\otimes n}\ltimes \mathbb Q[S_n]\to M_d(\mathbb Q)^{\otimes n}\otimes_{ \mathbb Q}\mathbb Q[\xi^{(i)}_{h,k}],\ x_i\mapsto \xi_i, \    S_n\mapsto S_n\subset             M_d(\mathbb Q)^{\otimes n}     $$ which factors through the algebra $\mathcal T_{X}^n(V)$. 

By Theorem \ref{Gol} we   have the Goldman element $\mathtt s_1 \in \mathbb Q[\xi_1,\ldots,\xi_k][c^{-1}]  ^{\otimes 2}$ and in the same way elements $\mathtt s_i \in \mathbb Q[\xi_1,\ldots,\xi_k][c^{-1}]  ^{\otimes n}$  which in the splitting  become the generators $(i,i+1)$  of the symmetric group  $S_n$.  

Therefore $\pi_n ( i,i+1]=\mathtt s_i$ (by uniqueness) and since by Formula \eqref{egg} the algebra $\mathbb Q[\xi_1,\ldots,\xi_k][c^{-1}]$  is closed under trace we have  $$\pi_n(\mathcal T_{X}^n(V)[c^{-1}])=  \mathbb Q[\xi_1,\ldots,\xi_k][c^{-1}]  ^{\otimes n}.$$  Since $\pi_n$ on $ \mathcal T_{X}^n(V)[c^{-1}]$ is injective we have the second claim.

 The fact that $\mathcal T_{X}^n(V)\subset M_d(\mathbb Q)^{\otimes n}\otimes_{ \mathbb Q}\mathbb Q[\xi^{(i)}_{h,k}]$ is a prime algebra follows from the fact that it is   torsion free over  its center  $ T^{(0)}_d\langle \Xi\rangle\subset \mathbb Q[\xi^{(i)}_{h,k}] $ and localizes to the Azumaya algebra  $ \mathbb Q[\xi_1,\ldots,\xi_k][c^{-1}]  ^{\otimes n}$  of rank $d^{2n}$  over its center a domain.
\end{proof}
\begin{remark}\label{inpr}
It is an interesting problem, cf.  \cite{TDN}, to understand formulas for the canonical element $\mathtt s$ as  a fraction of a tensor polynomial by a central element.  In the paper  \cite{TDN} the numerator of this expression is called a {\em swap polynomial}.  This problem is treated to some extent in the preprint \cite{P8} where we construct {\em balanced}  swap polynomials.
\end{remark}
\subsection{The spectrum}
Since the algebra $ \mathcal T_{X}^n(V) \subset M_{d^n}[\mathbb Q[\xi^{(i)}_{a,b }]$, Formula \eqref{imme},   is closed under trace then it is a  $d^n$--Cayley--Hamilton algebra, according to the Theory developed in \cite{ppr}. Its    trace algebra equals its center  $ T^{(0)}_d\langle \Xi\rangle$, the same for all $n$. 

Therefore, when we extend the scalars from $\mathbb Q$ to $\mathbb C$  the variety  with coordinate ring $ T^{(0)}_d\langle \Xi\rangle\otimes\mathbb C$  parametrizes isomorphism classes of  semisimple  representations of dimension $d^n$   of the algebra $ \mathcal T_{X}^n(V) \otimes\mathbb C$. 

Now  by the same reason, when $n=1$,   the variety  with coordinate ring $ T^{(0)}_d\langle \Xi\rangle\otimes\mathbb C$  parametrizes the  isomorphism classes of  semisimple  representations of dimension $d $ of the free algebra  $\mathbb C \langle \Xi\rangle$ or of the free $d$--Cayley--Hamilton algebra. That is,  if  $X$ has $k$ elements, conjugacy classes of  $k$--tuples of $d\times d$  matrices $(a_1,\ldots,a_k)\in  M_d(\mathbb C)^k$  generating a semisimple algebra. 

To this $k$--tuple is then associated a map  $ \mathcal T_{X}^n(V)\to M_d(\mathbb C)^{\otimes n}$  and the subalgebra of $M_d(\mathbb C)^k$ generated by the elements  $1^{\otimes i}\otimes a_j\otimes 1^{\otimes n-i-1}$ and $S_n$ is semisimple.  This matrix algebra  is the entire algebra $M_d(\mathbb C)^{\otimes n}$ if and only if  the representation is irreducible  which means it is a point of the spectrum of one of the Azumaya algebras $ \mathbb Q[\xi_1,\ldots,\xi_k][c^{-1}] $.  

 When the representation is semisimple and not irreducible we have a decomposition $V=\oplus_{j=1}^aV_j$ into irreducibles which induces a decomposition  $V^{\otimes n}=\oplus_{i_1,\ldots,i_n}V_{i_1}\otimes \ldots   \otimes V_{i_n}$  whose terms are permuted by the group $S_n$.  In each orbit there is a term 
$W:=V_1^{\otimes h_1}\otimes V_2^{\otimes h_2}\otimes \ldots \otimes V_k^{\otimes h_a},\ \sum h_i=n $  which is stabilized by a Young subgroup $H= S_ { h_1}\times S_ { h_1}\times \ldots \times S_ { h_a}$ giving the semisimple representation $Ind_H^{S_n}  W.$

%
%We   have the Goldman element $\mathtt s\in Q _d\langle \Xi\rangle\otimes Q _d\langle \Xi\rangle$
%We have   $M_d(\mathbb Q(\xi^{(i)}_{h,k}))\otimes_{\mathbb Q(\xi^{(i)}_{h,k})} M_d(\mathbb Q(\xi^{(i)}_{h,k})=M_d(\mathbb Q)^{\otimes 2}\otimes_{ \mathbb Q}     \mathbb Q(\xi^{(i)}_{h,k})$   and so  $\mathtt s=(1,2)\in  M_d(\mathbb Q)^{\otimes 2}.$
%
%and  the $GL(d)$ invariants
%\begin{equation}\label{glii}
%T\langle X\rangle^{\otimes n}\ltimes \mathbb Q[S_n]\to  \mathcal T_{X}^n(V) \to Q _d\langle \Xi\rangle\otimes_{Q^{(0)}_d\langle \Xi\rangle} Q _d\langle \Xi\rangle\otimes\ldots\otimes Q _d\langle \Xi\rangle\otimes_{Q^{(0)}_d\langle \Xi\rangle} Q _d\langle \Xi\rangle
%\end{equation}
%
%If $c\in \mathbb Q[\xi_1,\ldots,\xi_k]$ is a central polynomial (with no constant coefficient) it is known that
%$$  \mathbb Q[\xi_1,\ldots,\xi_k][c^{-1}]= T_d\langle \Xi\rangle[c^{-1}] $$ is an Azumaya algebra. Then, by uniqueness 
%$$ \mathtt s\in  T_d\langle \Xi\rangle[c^{-1}]\otimes_{T_d^{(0)}\langle \Xi\rangle[c^{-1}]}T_d\langle \Xi\rangle[c^{-1}]\subset Q _d\langle \Xi\rangle\otimes_{Q^{(0)}_d\langle \Xi\rangle} Q _d\langle \Xi\rangle$$
%            
\bigskip

  {\bf A comment}\quad 1)\quad Most of the results of this paper hold in a characteristic free way. In particular all identities with integer coefficients continue to hold.  
  Theorem \ref{FFT0} still holds, from the Theory of Donkin \cite{Don}, provided in
  Formula  \eqref{monf3} one replaces the factors $tr(N_i)$  by $\sigma_j(N_i)$. The Theorem of M. Artin has been generalized by Procesi to all rings, \cite{agpr}.

The only result  which should require a particular care is Theorem \ref{SFT}. 

In fact in order to carry out the proof in positive characteristic one would need to follow closely the rather difficult and non trivial calculations of Zubkov, see \cite{Zubkov1} or   \cite{depr0}. 

I have not tried to do this since it would have made the treatment extremely technical and very hard to follow but I believe that the argument can be generalized to this setting.\smallskip

2)\quad The algebra $T\langle X\rangle^{\otimes n}\ltimes \mathbb Q[S_n]$ contains the two subalgebras $T\langle X\rangle^{\otimes n}$  and $\mathbb Q\langle X\rangle^{\otimes n}.  $  The identities belonging to the first subalgebra are the {\em tensor trace identities}, the ones  belonging to the second subalgebra are the {\em tensor polynomial  identities}.  

Although it is true that these can be deduced from the antisymmetrizer their structure is far from being understood. 
                 A start in the study of   tensor polynomial  identities appears in the paper with F. Huber \cite{hp}. \smallskip

As for   tensor trace identities we know that for $n=1$  they are generated by the $d$--Cayley Hamilton identity. 

For  higher $n$ the situation is more complex since the algebra of equivariant maps is not flat over its trace algebra except for the special case $n=|X|=2$.  This is actually interesting and partially treated in \cite{P8}. Let us explain what happens in general.

 We  have a map $j: T\langle X\rangle^{\otimes n}\to Q _d\langle \Xi\rangle^{\otimes n} $. A simple argument shows that its Kernel is formed by the tensor trace identities since we may view this as a specialization to generic matrices.  Now when we use the  $d$--Cayley Hamilton identity this map factors through a map $j_d: T_d\langle X\rangle^{\otimes n}\to Q _d\langle \Xi\rangle^{\otimes n} $. 
 
 By simple localization arguments we then see that if  $a\in  T_d\langle X\rangle^{\otimes n}$ is in the Kernel of $j_d$ then there is an invariant $b$ so that $ba=0$.  In fact if $b$ is any central polynomial we have $b^ka=0$ for some $k$.

We   can also take   any nonzero discriminant  $\delta:=\det(tr(b_ib_j))$ of $d^2$ elements $b_i\in T_d\langle X\rangle$. Then after localizing $T_d\langle X\rangle[\delta^{-1}] $ is a free module with basis the $b_i$ over  the localized trace algebra $T_d^{(0)}\langle X\rangle[\delta^{-1}] $, so   its $n^{th}$ tensor power is also a free module and embeds in           $Q _d\langle \Xi\rangle^{\otimes n}.$

So we may say that up to multiplication by some power of this  discriminant a tensor trace identity can be deduced from the $d$--Cayley Hamilton identity.

On the other hand   the situation is similar to that of functional identities, so, as in the paper  \cite{bps},  one  can have       tensor trace identities  not deduced from the   $d$--Cayley Hamilton identity. For instance for $d=2$  the tensor polynomial identity  $St(2)=Alt_Xx_1x_2\otimes x_3x_4$, see \cite{hp}  is not a consequence of the $2$--Cayley Hamilton identity.

One can try to see which of the tensor polynomial identities discussed in \cite{hp}  are not a consequence of the $d$--Cayley Hamilton identity.

  \bibliographystyle{amsalpha}

\end{document}